\documentclass[12pt,reqno,a4paper]{amsart}%{amsart}%
\usepackage{amsmath,amssymb,amsthm,amsfonts}
\usepackage[UKenglish,USenglish]{babel}
\usepackage[dvips]{graphicx}
\usepackage{tikz}
\usepackage{mathrsfs}
\usepackage{enumitem}
\usepackage[a4paper, total={6in, 8in}, top=30mm, bottom=30mm]{geometry}
\usepackage{hyperref}
\theoremstyle{plain}
\newtheorem{thm}{Theorem}[section]
\newtheorem{lem}[thm]{Lemma}
\newtheorem{prop}[thm]{Proposition}
\newtheorem{cor}[thm]{Corollary}

\theoremstyle{definition}
\newtheorem{defn}{Definition}[section]

\newtheorem{rem}{Remark}[section]

\newcommand{\Rr}{\mathbb{R}}
\newcommand{\Nn}{\mathbb{N}}

\newcommand{\I}{\mathcal{I}}

\usepackage{soul, xcolor}

\usepackage[final,authormarkup=highlight]{changes}

\begin{document}

\setstcolor{red}

\title[Dynamics of piecewise contractions]
{Multi-dimensional piecewise contractions are asymptotically periodic}

\author[Gaiv\~ao]{Jos\'e Pedro Gaiv\~ao}
\thanks{J. P. Gaiv\~ao was partially supported by the Project CEMAPRE/REM - UIDB/05069/2020 - financed by FCT/MCTES through national funds.}
\address{Departamento de Matem\'atica and CEMAPRE, ISEG\\
Universidade de Lisboa\\
Rua do Quelhas 6, 1200-781 Lisboa, Portugal}
\email{jpgaivao@iseg.ulisboa.pt}

\author[Pires]{Benito Pires}\thanks{
B. Pires was partially supported by grants {\#}2019/10269-3 and {\#}2022/14130-2,
S\~ao Paulo Research Foundation (FAPESP)}
\address{Departamento de Computa\c c\~ao e Matem\'atica\\ Faculdade de Filosofia, Ci\^encias e Letras, Universidade de S\~ao Paulo\\ Ribeir\~ao Preto, SP, 14040-901, Brazil}
\email{benito@usp.br}

%\thanks{}
%\subjclass[2000]{...,...,...}
%\date{June 8, 2006}
\date{\today}  

\begin{abstract}
Piecewise contractions (PCs) are piecewise smooth maps that decrease distance between pairs of points in the same domain of continuity. The dynamics of a variety of systems is described by PCs. During the last decade, a lot of effort has been devoted to proving that in parametrized families of one-dimensional PCs, the $\omega$-limit set of a typical PC consists of finitely many periodic orbits while there exist atypical PCs with Cantor $\omega$-limit sets. In this article, we extend these results to the multi-dimensional case. More precisely,  we provide criteria to show that an arbitrary family $\{f_{\mu}\}_{\mu\in U}$ of locally bi-Lipschitz piecewise contractions $f_\mu:X\to X$ defined on a compact metric space $X$ is asymptotically periodic for Lebesgue almost every parameter $\mu$ running over an open subset $U$ of the $M$-dimensional Euclidean space $\mathbb{R}^M$. As a corollary of our results, we prove that piecewise affine contractions of $\Rr^d$ defined in generic polyhedral partitions are asymptotically periodic.
\end{abstract}

\maketitle
\tableofcontents

\section{Introduction}

Piecewise contractions (PCs) are piecewise smooth self-maps of a metric space that act on each of its continuity domains by decreasing the distance between pairs of points. Their importance stems from  countless dynamical systems whose dyna\-mics is governed by them: pseudo-billiards and switched server systems \cite{ABP22, BB04, CSR93, FP20}, outer billiards with contraction \cite{ MGU2015, G20, Jeong2, Jeong1}, polling systems \cite{MMPV06}, best-response dynamics in game theory \cite{GP21}, election methods \cite{JO19} and neural networks \cite{ECPG2013, Pires23}, just to name a few.  PCs also arise in subjects of pure mathematics such as non-trivial recurrence on surfaces \cite{CGut1983,CGBP2009,BS20}, uniform distribution of sequences modulo one \cite{B04}, $b$-ary expansion of real numbers \cite{ABP22,BP2020} and Hecke-Mahler series \cite{ BKLN21,LN18, MN21,GLN24}.

While a contraction acting on a compact or complete metric space always has a unique globally attractive fixed-point, a piecewise contraction with discontinuities may  have lots of locally attractive limit cycles and Cantor attractors coexisting side by side.

Thanks to a series of recent works, the dynamics of one-dimensional PCs (i.e., PCs of the interval with finitely many discontinuities) is by now well-understood. For instance, it is known that the attractor of locally injective PCs of the interval is the union of a finite number of periodic orbits with a finite number of Cantor sets \cite{CCG21}. However, the existence of Cantor attractors is rather exceptional. Indeed, in the measure-theoretical sense, a generic PC of the interval is asymptotically periodic, that is, the global attractor is formed by a finite number of periodic orbits \cite{NPR2014, NPR18, GN22}. Moreover, in some cases, the set of exceptional parameters where Cantor attractors exist has zero Hausdorff dimension \cite{gaivao2022hausdorff}. The following facts are also known: the subword complexity of  symbolic orbits of PCs of the interval is either uniformly bounded or a linear function \cite{CGM20, Pires19};  PCs of the interval with no periodic orbits and a Cantor attractor can be constructed  as suspension of any minimal interval exchange transformation \cite{Pires19}.  

A paradigmatic example of one-dimensional PC is the contracted rotation on $[0,1)$: $f_{\lambda,b}(x)=\lambda x + b \pmod{1}$ with  $0<\lambda<1$ and $1-\lambda<b<1$. Lifting $f_{\lambda,b}$ to $\Rr$, one can define a rotation number $\rho(f_{\lambda,b})\in[0,1)$ that describes the asymptotic average rotation of any orbit of $f_{\lambda,b}$. It is known that $f_{\lambda,b}$ has a Cantor attractor if and only if $\rho(f_{\lambda,b})$ is an irrational number. When $\rho(f_{\lambda,b})$ takes a rational value $\frac{p}{q}$ with co-prime positive integers $p$ and $q$, the contracted rotation $f_{\lambda,b}$ has a single global attracting periodic orbit with period $q$. Since the rotation number varies continuously and non-decreasingly with the parameter $b$, every rotation number in the interval $(0,1)$ is realizable by some $b\in(1-\lambda,1)$. Thus, the set of parameters $(\lambda,b)$, where $f_{\lambda,b}$ has Cantor attractors is uncountable, has Hausdorff dimension $\geq1$, but has zero Lebesgue measure \cite{B93,LN18}.  

In contrast to one-dimensional PCs,  less is known about the dynamics of multi-dimensional PCs.  In \cite{CGMU16},  it is proved that multi-dimensional PCs are asymptotically periodic whenever the attractor and the set of discontinuities are disjoint.  However, such sufficient condition turns out to be rather difficult to be verified in parametrized families of PCs, which are relevant for many applications.  In \cite{BD08}, it is proved that asymptotic periodicity is a generic property (in the measure-theoretical sense) of a family of piecewise-affine conformal contractions defined in $\mathbb{R}^2$  and sharing the same domains of continuity.  The family is obtained by vertical translation of the graph of any PC in the family.
Very recently,  it was proved in \cite{jain2024piecewise} that for a generic partition,  in a Baire residual sense,  the attractor of locally injective PCs on compact subsets of $\mathbb{R}^d$ is disjoint from the set of discontinuities,  implying that the attractor consists of a finite number of periodic orbits.  The authors also prove that PCs are generically topologically stable.  However, the results of \cite{jain2024piecewise} cannot be applied to parametrized families of multi-dimensional PCs. 

In this article,  we study genericity, from a measure-theoretical point of view, of asymptotic periodicity in parametrized families of multi-dimensional PCs. More precisely,  we provide a criteria (Theorem~\ref{th:main}) to show that an arbitrary family $\{f_{\mu}\}_{\mu\in U}$ of locally bi-Lipschitz piecewise contractions $f_\mu:X\to X$ defined on a compact metric space $X$ is asymptotically periodic for Lebesgue almost every parameter $\mu$ running over an open subset $U$ of the $M$-dimensional Euclidean space $\mathbb{R}^M$. Our criteria consists of two conditions: Hypothesis (E) in Definition \ref{vse} and Hypothesis (T) in Definition \ref{transversality}. We also provide Proposition \ref{lem26}, that makes it easier to verify Hypothesis (E). 

To show the strength and powerfulness of Theorem~\ref{th:main}, we give applications of it  to prove the almost sure asymptotic periodicity of two families of piecewise contractions $\{f_\mu\}_{\mu\in U}$ with varying domains of continuity. More precisely, Theorem~\ref{th:NPR2} shows that asymptotic periodicity is a generic property in families of one-dimensional piecewise contractions while Theorem~\ref{thm:multidim} and its corollary are multi-dimensional versions of the same result. Theorem~\ref{th:NPR2}  appeared previously in \cite[Theorem 1.4]{NPR18} but the proof presented therein does not generalize to higher dimensions. Such a generalization is one of the main contributions of this article.

The rest of the article is organized as follows. In Section \ref{general PCs}, we consider piecewise contractions defined on arbitrary metric spaces. First, we give a condition, Lemma \ref{lem:suff}, for a piecewise contraction to be asymptotically periodic. In the sequel, within the same section, we apply Lemma \ref{lem:suff} to a family of piecewise contractions, yielding a proof of  Theorem \ref{th:main}. Proposition \ref{lem26}, already mentioned in this introduction, is also stated and  proved in Section \ref{general PCs}. The results about one-dimensional PCs, including the proof of Theorem \ref{th:NPR2}, are given in Section \ref{onedimensionalpcs}. The proof of  Theorem \ref{thm:multidim} is given in  Section \ref{sec:proof multidim},  while its corollary is proved in Section \ref{proofs corollaries}.

\section{Main results}

To give the statement of our main results, we  first introduce two hypotheses: Hypothesis (E) and Hypothesis (T).  In order to do so, we need the formal definitions of piecewise contraction, regular point, asymptotic periodic self-map, itineraries of a piecewise contraction and family of piecewise contractions. 

Throughout this section, unless otherwise specified,
  $(X,d)$ will denote a compact metric space whose open balls are connected and  $\lambda \in (0,1)$.

\begin{defn}[Piecewise contraction]\label{def1}
A map $f\colon X\to X$ is called a \textit{piecewise $\lambda$-contraction on $X$} if there exist $N\in\Nn$, open, connected and pairwise disjoint subsets $A_1,\ldots, A_N$ of $X$ and bi-Lipschitz maps $\varphi_1,\ldots,\varphi_N\colon X\to X$ such that 
\begin{enumerate}
\item [(i)] $X':=\bigcup_{i=1}^N A_i$ is dense in $X$;
\item [(ii)] $f|_{A_i}=\varphi_i|_{A_i}$ for every $i\in\{1,\ldots,N\}$;
\item [(iii)] $\text{Lip}\,(\varphi_i)\leq \lambda$ for every $i\in\{1,\ldots,N\}$, where $\text{Lip}\,(\varphi_i)$ denotes the Lipschitz constant of $\varphi_i$.
\end{enumerate}

The set $\mathcal{A}=\{1,\ldots,N\}$ is called the \textit{label set} of $f$. The collection of sets $\{A_i\}_{i\in\mathcal{A}}$ is called the \textit{partition} of $f$ and the collection of maps $\{\varphi_i\}_{i\in\mathcal{A}}$ is called the \textit{iterated function system associated to $f$}. The \textit{singular set} of $f$ is 
$$S(f):=X{\setminus} X'.$$

\end{defn}

\begin{defn}[Regular point]\label{regularp}
Given a piecewise $\lambda$-contraction $f:X\to X$ and $n\in\Nn$,  we say that $x\in X$ is a \textit{regular point of order $n$} if $f^j(x)\notin S(f)$ for all $0\le j<n$. We say that $x\in X$ is a \textit{regular point} if $x$ is a regular point of order $n$ for all $n\in\Nn$. A periodic orbit is called  \textit{regular} if one (and therefore all) of its periodic points is regular.  We denote by $Z(f)$ the set of regular points of $f$.
\end{defn}

Given a self-map $f\colon X\to X$, we denote by
$$ \omega(f,x) := \bigcap_{m\ge 1} \overline{\bigcup_{n\ge m} {\left\{f^n(x)\right\}}}
$$
the $\omega$-limit set of a point $x\in X$ under $f$. 

Given a forward invariant set $Y\subset X$, the union of $\omega(f,x)$ over all points $x\in Y$ is denoted by $\omega(f,Y)$, i.e.,
$$ 
\omega(f,Y) := \bigcup_{x\in Y}  \omega(f,x).
$$
When $Y=X$, we simply write $\omega(f):=\omega(f,X)$.

\begin{defn}[Asymptotically periodic]\label{def asymptotically periodic}
Given a self-map $f\colon X\to X$ and a forward invariant set $Y$, we say that $f$ is \textit{asymptotically periodic on $Y$} if $\omega(f,Y)$ is the union of finitely many periodic orbits.  Moreover, $f$ is called \textit{asymptotically periodic} if $Y=X$.
\end{defn}

\begin{rem}\label{rem1} If $\omega(f,Y)$ is finite and $\omega(f,Y)\cap S(f)=\emptyset$, then it must be a union of finitely many periodic orbits because, in such a case,  
$f$ takes  $\omega(f,Y)$ onto itself. Moreover, $f$ is asymptotically periodic on a closed forward invariant set $Y\subset X$ if and only if $f\vert_{Y}:Y\to Y$ is asymptotically periodic. Finally,   if $f$ is asymptotically periodic on $Y$ and $X=\bigcup_{n\geq0}f^{-n}(Y)$, then $f$ is asymptotically periodic. 
\end{rem}

%\begin{exmp}
%Let $0<\lambda<1$ and $1-\lambda <b < 1$. The \textit{contracting rotation} is the map $R_{\lambda,b}:[0,1]\to[0,1]$ defined by $$R_{\lambda,b}(x)=\begin{cases}\lambda x+ b& \textrm{if}\quad x<(1-b)/\lambda\\
%\lambda x+ b-1& \textrm{if} \quad  x\geq (1-b)/\lambda
%\end{cases},$$
%which is a piecewise $\lambda$-contraction on $X=[0,1]$ with iterated function system $\{\varphi_i\colon[0,1]\to[0,1]\}_{i=1,2}$ given by $\varphi_1(x)=\lambda x+b$ and $\varphi_2(x)=\lambda x + b -1$.
%The singular set and  partition of $R_{\lambda,b}$ are $$S(R_{\lambda,b})= \{0, (1-b)/\lambda, 1 \},\quad A_1=\big(0, (1-b)/\lambda\big),\quad A_2=\big((1-b)/\lambda,1 \big).$$
%It has been proved by many authors that for Lebesgue almost every $(\lambda,b)$, the contracting rotation $R_{\lambda,b}$  is asymptotically periodic. Moreover, there exist exceptional parameter vectors $(\lambda,b)$ for which $\omega\big(R_{\lambda,b}\big)$ is a Cantor set (see, for instance, \cite{B93,LN18}). 
%\end{exmp}

\begin{defn}[Itineraries]\label{def itineraries}  Let $f:X\to X$ be a piecewise $\lambda$-contraction with label set $\mathcal{A}= \{1,\ldots, N\}$ and partition $\{A_i\}_{i\in\mathcal{A}}$. Given any regular point $x\in Z(f)$, we call \textit{itinerary of $x$} the sequence $i(x)=(i_0,i_1,\ldots)\in\mathcal{A}^{\Nn_0}$, $\Nn_0 =\{0\}\cup \Nn$, uniquely defined by the condition $f^n(x)\in A_{i_n}$ for every $n\geq0$. Likewise,  if $x$ is a regular point of order $n$ for some $n\geq1$, then we call  the $n$-tuple $(i_0,i_1,\ldots,i_{n-1})\in\mathcal{A}^{n}$ the \textit{itinerary of order $n$ of $x$} if $f^j(x)\in A_{i_j}$ for all $0\le j<n$. The set of all itineraries of order $n$ is denoted by $\I_n(f)$.
\end{defn}

Given $M\ge 1$, let $U\subset\mathbb{R}^M$ be a set of positive Lebesgue measure and $\mu^*\in U$.

\begin{defn}[Family of piecewise $\lambda$-contractions]\label{def2}
We say that $\{f_\mu\}_{\mu\in U}$ is a \textit{family of piecewise $\lambda$-contractions} on $X$ if there exists $N\in\Nn$ such that the following conditions are satisfied for each $\mu\in U$:
\begin{enumerate}
\item [(C1)] $f_\mu$ is a piecewise $\lambda$-contraction on $X$;
\item [(C2)] the label set of $f_\mu$ is equal to $\mathcal{A}=\{1,\ldots,N\}$.
\end{enumerate}
\end{defn}

Let $\{f_\mu\}_{\mu\in U}$ be a family of piecewise $\lambda$-contractions on $X$. We denote by $S_\mu:=S(f_\mu)$ the singular set of $f_\mu$ and by $Z_\mu=Z(f_\mu)$ the set of regular points of $f_\mu$.

Given $\delta>0$, set 
\begin{equation}\label{J}
U_{\delta}(\mu^*) :=U\cap B_\delta(\mu^*), \quad \mathcal{J}_n^\delta(\mu^*) :=\bigcup_{\mu\in U_{\delta}(\mu^*)} \I_n(f_\mu),
\end{equation} 
where $B_\delta(\mu^*)$ denotes the open ball of radius $\delta>0$ centred at $\mu^*\in\Rr^M$ and $\I_n(f_\mu)$ is the set of all itineraries of order $n$ of $f_\mu$ as in Definition~\ref{def itineraries}.
By Condition (C2), $\mathcal{J}_n^\delta(\mu^*)$ is a finite set whose cardinality satisfies $\#\mathcal{J}_n^\delta(\mu^*)\le N^n$. 
\bigskip

\begin{defn}\label{vse}
We say  $\{f_\mu\}_{\mu\in U}$  satisfies \textit{Hypothesis (E)} at $\mu^*\in U$ if
\begin{equation}\tag{E}
\lim_{\delta\to0^+}\limsup_{n\to\infty}\frac1n\log\#\mathcal{J}_n^\delta(\mu^*) =0.
\end{equation}
\end{defn}

In what follows, we denote by $\{\varphi_{i,\mu}\}_{i\in\mathcal{A}}$ the iterated function system associated to the piecewise $\lambda$-contraction $f_{\mu}$, as in Definition \ref{def1}.  Given an $n$-tuple $\alpha=(i_0,i_1,\ldots,i_{n-1})\in\mathcal{A}^n$, we define the map
\begin{equation}\label{def varphi}
\varphi_\mu^{\alpha}:=\varphi_{i_{n-1},\mu}\circ\varphi_{i_{n-2},\mu}\circ\cdots\circ\varphi_{i_{1},\mu}\circ \varphi_{i_{0},\mu}.
\end{equation}

Denote by $\text{Leb}^*$ the Lebesgue outer measure.
%Moreover, given $\varepsilon>0$, define
%$$
%S_{f_\mu}^\varepsilon:=\{x\in X\colon d(x,S_{f_\mu})\leq \varepsilon\}.
%$$
\begin{defn}\label{transversality} We say that $\{f_\mu\}_{\mu\in U}$  satisfies \textit{Hypothesis (T)} at $\mu^*\in U$
 if there exist $x_0\in X$, $\varepsilon_0>0$, $\delta_0>0$, $n_0\geq 1$, $a>0$ and $c>0$ such that for every $0<\varepsilon<\varepsilon_0$, $n\geq n_0$ and $\alpha\in \mathcal{J}_n^{\delta_0}(\mu^*)$,
\begin{equation}\tag{T} 
\text{Leb}^*\left(\left\{\mu\in U_{\delta_0}(\mu^*)\colon d(\varphi_\mu^{\alpha}(x_0),S_{\mu}) \leq \varepsilon   \right\}\right) \leq c\, \varepsilon^a.
\end{equation}
\end{defn}

We are now ready to state our first result.

\begin{thm}\label{th:main} Let $(X,d)$ be a compact metric space whose open balls are connected and  let $\lambda \in (0,1)$.
If $U\subset\mathbb{R}^M$ is a set of positive Lebesgue measure and
  $\{f_\mu\}_{\mu\in U}$ is a family of piecewise $\lambda$-contractions on X satisfying the hypotheses (E) and (T) at $\mu^*\in U$, then there is $\delta>0$ such that $f_\mu$ is  asymptotically periodic on $Z_\mu=Z(f_{\mu})$ for Lebesgue almost every $\mu\in U_\delta(\mu^*)$.
\end{thm}

The proof of Theorem~\ref{th:main} is given in Subsection~\ref{proof:main}.  Let us briefly explain the strategy of the proof.  Hypothesis (E) guarantees that the number of distinct itineraries of $f_{\mu^*}$ grows sub-exponentially with $n$ and that this property is robust. Since,  for any $\mu$ that is  $\delta$-close to $\mu^*$, the $\omega$-limit set $\omega(f_{\mu},Z_\mu)$ can be coded using the itineraries in $\mathcal{J}_n^\delta(\mu^*)$, which grow sub-exponentially with $n$, the contraction together with Hypothesis (T) ensures that the measure of parameters where $\omega(f_{\mu},Z_\mu)$ gets too close to the singular set of $f_\mu$ is arbitrarily small. Hence, for almost every parameter $\mu$ sufficiently close to $\mu^*$, the set $\omega(f_{\mu},Z_\mu)$ is disjoint from the singular set, which implies that it is formed by a finite number of periodic points.  

In concrete families,  it is much more difficult to verify Hypothesis (E) than Hypothesis (T).  For this reason,  in Proposition~\ref{lem26}, we provide a sufficient condition that guarantees Hypothesis (E) and that it is easier to check.

Next, we will use Theorem~\ref{th:main} to deduce the asymptotic periodicity of generic multi-dimensional piecewise-affine contractions.  Let $d\ge 1$ be an integer. In what follows, we will now suppose that $X\subset\mathbb{R}^d$ is a \textit{compact convex polytope}, i.e., $X$ is a compact set that is the intersection of finitely many closed half-spaces whose bounding hyperplanes is the family 
\begin{equation}\label{fdhyp}
\mathcal{F}=\{F_1,\ldots, F_{r}\}, \quad\textrm{where}\quad
F_j = \{x\in\mathbb{R}^d: \langle u^{(j)},x\rangle = c_j\},
\end{equation}
where $u^{(j)}$ is a unit vector in $\mathbb{R}^d$,  $c_j\in\mathbb{R}$, and  $\langle\cdot,\cdot\rangle$ denotes the Euclidean inner product in $\mathbb{R}^d$.  In matrix notation, we represent the polytope as 
\begin{equation}\label{defX}
X=\{x\in\Rr^d\colon Mx\leq c\},
\end{equation} where $M$ is the $r\times d$ matrix whose $i$-th row,  $1\le i\le r$, is formed by the entries of the unit vector $u^{(i)}$  and $c=(c_1,\ldots, c_r)^t$. The polytope $X$  need  not be full-dimensional, i.e., it could be contained in a proper affine subspace of $\Rr^d$.
In low dimensions, we can think of $X$ as a convex polygon in $\mathbb{R}^2$, a 2-simplex in $\Rr^3$ or a compact convex polyhedron in $\mathbb{R}^3$. In the sequel, a subset $A$ of $X$ is open in $X$ in the subspace topology.

\begin{defn}[Piecewise-affine contraction]
We say that $f\colon X\to X$ is a \textit{piecewise-affine contraction} on $X$ if there are finitely many disjoint open polytopes $A_1,\ldots, A_N$ contained in $X$ with $N\in\Nn$ such that 
\begin{enumerate}
\item $X':=\bigcup_{i=1}^N A_i$ is dense in $X$;
\item there is some norm $\|\cdot\|$ in $\mathbb{R}^d$ such that for every $i\in\{1,\ldots,N\}$,  the map $f|_{A_i}$ can be extended to an injective affine contraction on $\mathbb{R}^d$.
\end{enumerate}
\end{defn}

Clearly,  a piecewise-affine contraction $f$  is a piecewise $\lambda$-contraction (Definition~\ref{def1}) where $\lambda=\max_i\|Df|_{A_i}\|$.  

Now, we introduce a family of piecewise-affine contractions on $X$.  Let $\ell\ge 1$, $\mathcal{A}=\{-1,1\}^\ell$ and $\varphi_i: \mathbb{R}^d\to \mathbb{R}^d$, $i\in \mathcal{A}$,  be injective affine contractions that take $X$ into its relative interior, i.e.,  let
$\Lambda_i\in\textrm{GL}_d$ with $\Vert \Lambda_i\Vert < 1$\footnote{$\Vert \Lambda_i\Vert = \sup\, \{\Vert \Lambda_i x\Vert: \Vert x \Vert=1 \}$ denotes the operator norm induced by some norm $\Vert \cdot\Vert$ in $\mathbb{R}^d$.}  and $b_i\in \mathbb{R}^d$, $i\in\mathcal{A}$, be such that
\begin{align}
\varphi_i(x) &= \Lambda_i x + b_i, \quad \forall x\in \mathbb{R}^d,\quad \forall i\in\mathcal{A}, \label{gld}\\
\varphi_i(X) &\subset \mathrm{relint}\,(X),\quad \forall i\in\mathcal{A}, \label{smap}
\end{align}
where $\textrm{relint}\,(X)=\{x\in \Rr^d\colon \exists\,\epsilon>0,\, B_\epsilon(x)\cap \text{aff}(X)\subset X\}$, $B_\epsilon(x)$ denotes the open ball centred at $x$ of radius $\epsilon$ and $\text{aff}(X)$ is the affine hull of $X$.  We will use the IFS $\Phi=\{\varphi_i\}_{i\in\mathcal{A}}$ to define piecewise-affine contractions whose singular set is determined by the  boundary of $X$ and
$\ell$ hyperplanes in $\mathbb{R}^d$.  More precisely,  let
\begin{equation}\label{Hj}
H_j\left(\mu\right) := \left\{  {x}\in \mathbb{R}^d:\langle  {v}^{(j)}, {x}\rangle = \mu_j\right\},\quad  1\le j\le \ell,
\end{equation}
where $\mu=(\mu_1,\ldots,\mu_{\ell})\in \mathbb{R}^\ell$,  $v^{(1)},\ldots, v^{(\ell)}\in \mathbb{R}^d$ are unit vectors in $\Rr^d$.
Let
\begin{equation}\label{Hj2}
\mathscr{H}_\mu :=\{H_j(\mu)\colon 1\leq j\leq \ell\},\quad \mu\in \Rr^\ell.
\end{equation}

\begin{defn}[Label map] \label{def:label map}  Given $\mu\in\mathbb{R}^\ell$,
 the \textit{label map} with respect to the family $\mathscr{H}_\mu$ of hyperplanes is the map $\sigma_\mu\colon \Rr^d\to \mathcal{A}$ defined by $$\sigma_\mu(x)=\big(s_1(x),\ldots,s_\ell(x)\big),$$ 
where\footnote{The choices $s_j(x)=-1$ when $\langle {v}^{(j)}, {x}\rangle = \mu_{j}$ are arbitrary and not relevant for the proof of the main result in this section.}
$$ s_j(x) = \begin{cases} -1 & \textrm{if} \quad \added{\langle {v}^{(j)}, {x}\rangle} \le \mu_j,\\
\phantom{-}1 & \textrm{if} \quad \langle {v}^{(j)}, {x}\rangle > \mu_{j}
\end{cases},\quad 1\le j\le \ell.
$$  
\end{defn}
Finally, given $\mu\in\mathbb{R}^\ell$, let $f_\mu: X\to X$ be defined by
\begin{equation}\label{fmu40}
f_\mu(x)=\varphi_{\sigma_\mu(x)}(x), \quad \forall\, x\in X.
\end{equation}
Clearly,  for each $\mu\in \mathbb{R}^\ell$,  the map $f_\mu$ is a piecewise-affine contraction.  
 %and
%\begin{equation}\label{V333}
% V =\bigcup_{k=1}^{n}\bigcup_{\alpha\in\mathcal{B}^k} \{v\in\mathbb{R}^d{\setminus}\{0\}: v \textrm{  is not eigenvector of  } \Lambda^{\alpha}\}. 
%\end{equation}
%\begin{defn}[singular connection] Let $\Phi=\{\varphi_i\}_{i\in\mathcal{B}}$, where each $\varphi_i$ is an invertible affine contraction of $\mathbb{R}^d$ satisfying \eqref{gld} and \eqref{smap}.
%Given vectors $v^{(1)},\ldots, v^{(M)}\in\mathbb{R}^d$,
%we say that  $(\Phi,v^{(1)}, \ldots, v^{(M)})$ has a \textbf{singular connection} if there are $n\geq 1$ and $\alpha\in\mathcal{B}^n$ such that $\{\Lambda^\alpha v^{(i)}, v^{(j)}\}$ is linearly dependent for some $(i,j)\in\{1,\ldots,M\}^2$. 
%\end{defn}
We are now ready to state our second main result.

\begin{thm}\label{thm:multidim}Let $f_\mu$ be the piecewise affine-contraction defined in \eqref{fmu40}.  The map $f_\mu$ is asymptotically periodic on $Z_\mu$ for Lebesgue almost every $\mu\in\Rr^\ell$. 
\end{thm}

We prove Theorem~\ref{thm:multidim} in Section~\ref{sec:proof multidim}. 
From this result, we deduce the following corollary whose proof can be found in Section~\ref{proofs corollaries}. 

%\begin{cor}\label{cor:1}
%Suppose that, for each $i\in\mathcal{A}$,  the map $\varphi_i$ in \eqref{gld} is conformal. Then, for Lebesgue almost every $\mu\in \Rr^\ell$,  the map $f_\mu$ is asymptotically periodic on $Z_\mu$.
%\end{cor}
%
%\begin{cor}\label{cor:2}
%Suppose that the matrices $\{\Lambda_i\}_{i\in\mathcal{A}}$  in \eqref{gld} commute.  For Lebesgue almost every $\mu\in \Rr^\ell$,  the map $f_\mu$ is asymptotically periodic on $Z_\mu$.
%\end{cor}

\begin{cor}\label{cor:3}
Suppose that, for each $i\in\mathcal{A}$, the affine contraction $\varphi_i$ in \eqref{gld} is a homothety of $\Rr^d$. For Lebesgue almost every $\mu\in \Rr^\ell$,  the map $f_\mu$ is asymptotically periodic and its periodic points are regular.
\end{cor}

\section{Piecewise contractions on metric spaces}\label{general PCs}

Let  $(X,d)$ be a compact metric space whose open balls are connected. 
Denote by $\mathrm{diam}(A) = \sup\{d(x,y)\colon x,y\in A\}$ the diameter of a set $A\subset X$. We also suppose that $\mathrm{diam}(X)>0$.
Throughout this section,  $f:X\to X$ is a piecewise $\lambda$-contraction on  $X$ with label set $\mathcal{A}= \{1,\ldots, N\}$,  partition $\{A_i\}_{i\in\mathcal{A}}$ and $\{\varphi_i\}_{i\in\mathcal{A}}$ denotes the iterated function system associated to $f$.

Recall that $Z(f)$ denotes the set of regular points of $f$ (Definition \ref{regularp}). The set of regular points of $f$ equals 
\begin{equation}\label{def Z}
Z(f):= X'\cap f^{-1}(X')\cap f^{-2}(X')\cap \cdots.
\end{equation}
Moreover,  $Z(f) = \bigcap_{n\geq 1} Z_n(f)$, where $Z_n(f)=X'\cap f^{-1}(X')\cap\cdots f^{-n+1}(X')$,  $n\geq1$,  and $Z(f)$  is a forward invariant residual subset of $X$.

\begin{lem}[cf. Lemma 6.1 in \cite{CGMU16}]
$Z(f)$ is a residual subset of $X$. 
\end{lem} 
\begin{proof}
To simplify the notation, let $Z=Z(f)$ and $Z_n=Z_n(f)$.  Clearly, $Z_{n+1}=X'\cap f^{-1}(Z_{n})$.  By definition,  $Z_1 = X'$  is open and dense.  Using induction in $n$ we can show that $Z_n$ is open and dense for every $n\geq1$. Indeed, suppose that $Z_{n}$ is open and dense. Then, $\varphi_i^{-1}(Z_{n})$ is open and dense because $\varphi_i$ is bi-Lipschitz. Thus, $Z_{n+1}=X'\cap f^{-1}(Z_{n}) = \bigcup_i A_i\cap f^{-1}(Z_{n}) = \bigcup_i (f|_{A_i})^{-1}(Z_{n})=\bigcup_i (\varphi_i|_{A_i})^{-1}(Z_{n})=\bigcup_i A_i\cap \varphi_i^{-1}(Z_{n})$ is also open and dense.  Since $Z=\bigcap_{n\geq0} Z_n$, we conclude that $Z$ is a countable intersection of open and dense subsets of $X$.
\end{proof}

As in \eqref{def varphi}, given a $n$-tuple $\alpha=(i_0, i_1,\ldots, i_{n-1})$ in $\mathcal{A}^n$, we define the map
\begin{equation}\label{4344}
\varphi^\alpha = \varphi_{i_{n-1}} \circ\varphi_{i_{n-2}}\circ \cdots \circ \varphi_{i_1}\circ\varphi_{i_0}.
\end{equation}

Fix $x_0\in X$ and set 

\begin{equation}\label{omegaf}
\Omega(f) :=\bigcap_{m\geq1}\overline{\bigcup_{n\geq m}\{\varphi^\alpha(x_0)\colon \alpha\in\mathcal{I}_n(f)\}}.
\end{equation}

The following couple of lemmas are elementary. For the convenience of the reader we include a proof.
\begin{lem}
$\omega(f,x)\subseteq \Omega(f)$ for every $x\in Z(f)$.
\end{lem}

%\begin{proof} 
% Let $y\in \omega(f)$, then there exists a sequence $n_j\nearrow\infty$ and a  point $z\in X$ such that $f^{n_j}(z)\to y$ as $j\to\infty$. By (iv) in Definition \eqref{def1}, there exists $\alpha^{(j)} = (i_0, i_1,\ldots, i_{n_j-1}) \in \{1,\ldots, k\}^{n_j}$ such that $f^n(x)=\varphi_{i_n}(x)$ for all $0\le n\le n_j-1$.  Define \textcolor{red}{$y_j:=\varphi^{\alpha^{(j)}}(x)$.} Then $$d(\textcolor{red}{f^{n_j}(z)},y_j)=d\textcolor{red}{(\varphi^{\alpha^{(j)}}(z),\varphi^{\alpha^{(j)}}(x))}\leq \lambda^{n_j} d(z,x)\leq \lambda^{n_j}\text{diam}(X). 
%$$
%Thus, $y_j\to y$ as $n_j\to\infty$, i.e., $y\in \textcolor{red}{\Omega(f)}$.
% \end{proof}

\begin{proof} 
 Let $x\in Z(f)$ and $y\in \omega(f,x)$, then there exists a sequence $n_j\nearrow\infty$ such that $f^{n_j}(x)\to y$ as $j\to\infty$. Let $\alpha^{(j)}\in \mathcal{I}_{n_j}(f)$ denote the itinerary of order $n_j$ of $x$. Then, $f^{n_j}(x)=\varphi^{\alpha^{(j)}}(x)$.  Define $y_j:=\varphi^{\alpha^{(j)}}(x_0)$. Then $$d(f^{n_j}(x),y_j)=d(\varphi^{\alpha^{(j)}}(x),\varphi^{\alpha^{(j)}}(x_0))\leq \lambda^{n_j} d(x,x_0)\leq \lambda^{n_j}\text{diam}(X). 
$$
Thus, $y_j\to y$ as $j\to\infty$, i.e., $y\in \Omega(f)$.
 \end{proof}
 
 Recall that $\mathcal{I}_n(f)$ is the set of itineraries of order $n$ of $f$ (Definition~\ref{def itineraries}).
\begin{lem}\label{lem:covering} Given $n\ge 1$, the set $\Omega(f)$ can be covered by a finite union of closed balls of radius $2\text{diam}(X)\lambda^n$ centred at the points $\varphi^\alpha(x_0)$ with $\alpha\in \mathcal{I}_n(f)$.
\end{lem}

\begin{proof}
Given $y\in \Omega(f)$, there are $n_j\nearrow\infty$ and $ \alpha^{(j)}\in\mathcal{I}_{n_j}(f)$ such that $y_j:=\varphi^{\alpha^{(j)}}(x_0)\to y$ as $j\to\infty$. Take $j$ sufficiently large such that $n_j\geq n$ and $d(y,y_j)\leq \text{diam}(X) \lambda^{n}$. Denote by $[\alpha^{(j)}]$ the last $n$ entries of $\alpha^{(j)}$, then $[\alpha^{(j)}]\in \mathcal{I}_n(f)$. Moreover, 
$$
d(y_j, \varphi^{[\alpha^{(j)}]}(x_0) )=d(\varphi^{\alpha^{(j)}}(x_0),\varphi^{[\alpha^{(j)}]}(x_0) )\leq \text{diam}(X) \lambda^n,$$ which gives $d(y,\varphi^{[\alpha^{(j)}]}(x_0))\leq 2\text{diam}(X)\lambda^n$ by the triangle inequality.
\end{proof}

\begin{defn}\label{i-ii} We say that a collection $\mathscr{C} = \{C_j\}_{j\in\Lambda}$ of finitely many pairwise disjoint  subsets of $X$ is \textit{strongly invariant} if there exists $n_0\ge 1$ such \added{that} the following hypotheses are satisfied: 
\setlist[itemize]{leftmargin=8mm}
\begin{itemize}
\item [$(i)$] for each $A\in\mathscr{C}$, there is $A'\in\mathscr{C}$ such that  ${f(A)}\subset A'$;
\item [$(ii)$] for each $B\in\mathscr{C}$ and $n\ge n_0$, there is $B'\in\mathscr{C}$ such that  $\overline{f^n(B)}\subset B'$;
\item [$(iii)$] for each $C\in\mathscr{C}$, there is $A\in \{A_1,\ldots, A_N\}$ such that $C\subset A$.
\end{itemize}
\end{defn}
\begin{lem}\label{lem333} If  
$\mathscr{C}=\{C_j\}_{j\in \Lambda}$ is a strongly invariant collection of pairwise disjoint open subsets of $X$, then there exist  periodic orbits $\gamma_1,\ldots, \gamma_r$ of $f$ in $\cup_{j\in \Lambda} C_j$ such that  $\omega(f,y)\in \{\gamma_1,\ldots,\gamma_r\}$ for each $y\in \cup_{j\in \Lambda} C_j$.
\end{lem}
\begin{proof} Since $\Lambda$ is finite in Definition \ref{i-ii}, we may assume that $\Lambda=\{1,\ldots, q\}$. Let $j_0\in \Lambda$. 
Since $\mathscr{C}=\{C_j\}_{j\in \Lambda}$ is a strongly invariant collection, by item (i) of Definition \ref{i-ii},
 there exists $\alpha=(j_0, j_1\ldots)\in \{1,\ldots,q\}^{\mathbb{N}}$ such that ${f\big(C_{j_n}\big)} \subset C_{j_{n+1}}$ for all $n\ge 0$.   
Moreover, there exist $m\ge 0$ and  $p\ge 1$ such that $j_m = j_{m+p}$. Then,
\begin{equation}\label{fpcim1}
 {f^p\big(C_{j_m}\big)} \subset C_{j_{m+p}} = C_{j_m}.
\end{equation}
Let $n_0\ge 1$ be as in Definition \ref{i-ii}. By induction, \eqref{fpcim1} implies
$
 {f^n\big(C_{j_m}\big)} \subset C_{j_m}
$
for   $n:=n_0p$.
In this way, by item (ii) of Definition \ref{i-ii}, since the sets in $\{C_j\}_{j\in\Lambda}$ are pairwise disjoint, we conclude that
\begin{equation}\label{overline} 
\overline{f^n(C_{j_m})}\subset C_{j_m}.
\end{equation}
By item (iii) of Definition \ref{i-ii}, for each $j\in \Lambda$, there exists a unique
$\tau(j) \in\mathcal{A}$ such that $C_{j}\subset A_{\tau(j)}$, where $\mathcal{A}=\{1,\ldots,N\}$ and $\{A_i\}_{i\in\mathcal{A}}$ denote, respectively, the label set and the partition of $f$. Hence, if we set $i_m= \tau (j_m), \ldots, i_{m+n-1} = \tau (i_{m+n-1})$, we obtain
\begin{equation}\label{fpcim3}
f^n (C_{j_m}) = \big(\varphi_{i_{m+n-1}} \circ \cdots\circ \varphi_{i_{m+1}}\circ\varphi_{i_m}\big)(C_{j_m}).
\end{equation}
Combining \eqref{overline},  \eqref{fpcim3} and the continuity of  $g=\varphi_{i_{m+n -1}} \circ \cdots\circ \varphi_{i_{m+1}}\circ\varphi_{i_m}$ leads to
$$ 
g\big(\overline{C_{j_m}}\big) \subset \overline{f^n(C_{j_m})}\subset C_{j_m}.
$$
Since $\overline{C_{j_m}}\subset X$ and $X$ is compact, we have that $\overline{C_{j_m}}$ is complete. Hence, by applying the Banach Fixed-Point Theorem to the contraction $g$ on $\overline{C_{j_m}}$, we conclude that there exists a unique fixed point  of $g$ in $C_{j_m}$. In terms of $f$, this means that there exists a unique periodic orbit $\gamma$ in $\cup_{\ell=m}^{m+n-1} C_{j_\ell}$ such that $\omega(f,y)=\gamma$ for all $y\in \cup_{\ell=0}^{m+n-1} C_{j_\ell}$. The proof is concluded by repeating this procedure finitely many times until exhausting all the sets in $\mathscr{C}$.      
\end{proof}

Denote by $\mathscr{C}^{(1)}=\{A_i\}_{i\in \mathcal{A}}$ the partition of $f$, and let $\mathscr{C}^{(n)}$ with $n\geq 2$ be the collection of non-empty open sets 
\begin{equation}\label{Qna}
A^{\alpha} =A_{i_0}\cap\varphi_{i_0}^{-1}(A_{i_1})\cap\cdots\cap (\varphi_{i_{n-2}}\circ\cdots\circ\varphi_{i_0})^{-1}(A_{i_{n-1}})
\end{equation}
with $\alpha=(i_0,i_1,\ldots,i_{n-1}) \in \mathcal{I}_n(f)$. Each set $A^{\alpha}\in \mathscr{C}^{(n)}$ with $\alpha\in\mathcal{I}_n(f)$  consists of all regular points of order $n$ with itinerary of order $n$ equal to $\alpha$. Let $\ell\colon X\to \Nn\cup\{+\infty\}$ be defined by
$$
\ell(x)=\inf\{n\geq1\colon x\text{ is not a regular point of order $n$}\}.
$$

Using the convention $\inf\emptyset=+\infty$, we have $\ell(x)=+\infty$ iff $x$ is a regular point. Moreover,  $\ell(x)=1$ iff $x\in S(f)$.  Otherwise, $1<\ell(x)<+\infty$ iff $f^{\ell(x)-1}(x)\in S(f)$ and there is $\alpha=(i_0,i_1,\ldots,i_{\ell(x)-2})\in \mathcal{I}_{\ell(x)-1}(f)$ such that $f^n(x)\in A_{i_n}$ for every $0\leq n<\ell(x)-1$.

\begin{lem}\label{lem:suff1} If $\Omega(f)\cap S(f)=\emptyset$, then there is $n\in \Nn$ such that all regular points of order $n$ are regular points and $\mathscr{C}^{(n)}$ is strongly invariant. 
\end{lem}

\begin{proof}
Since both sets $\Omega(f)$ and $S(f)$ are compact,  by hypothesis, there exist $n_0>0$ and $\varepsilon>0$ such that
%$$
%\text{dist}\left(\bigcup_{n\geq N_1}\{H_\omega(\lambda)\colon \omega\in \I_n^\varepsilon\},\{a_1,\ldots,a_{k-1}\}\right)>\delta.
%$$
\begin{equation}\label{lem:hypothesis}
d( \varphi^\alpha(x_0),z)\geq \varepsilon,\quad \forall\,n\geq n_0,\,\forall \alpha\in \mathcal{I}_n(f),\,\forall z\in S(f).
\end{equation}
Let $n_1\ge \max\, \{n_0, 2\}$ be such that 
$$\text{diam}(X)\lambda^n< \frac{\varepsilon}{2}\quad \textrm{for all} \quad n\ge n_1.$$ 

Now we show that all regular points of order $n_1$ are in fact regular points, i.e., 
\begin{equation}\label{lem:claim 2} 
Z(f) = \bigcup_{ \alpha \in \mathcal{I}_{n_1}(f)} A^{\alpha}.
\end{equation}
Indeed,  suppose by contradiction that there is a regular point $y$ of order $n_1$ that is not a regular point of order $n_2$ for some $n_2>n_1$.   Then $n_1<\ell(y)<+\infty$ and  there is $\alpha\in \mathcal{I}_{\ell(y)-1}(f)$ such that $z:=f^{\ell(y)-1}(y)=\varphi^\alpha(y)\in S(f)$ and
$$
d(\varphi^\alpha(x_0),z) = d(\varphi^\alpha(x_0),\varphi^\alpha(y))\leq \text{diam}(X)\lambda^{\ell(y)-1}<\frac{\varepsilon}{2},
$$
which contradicts \eqref{lem:hypothesis}.

It remains to show that $\mathscr{C}^{(n_1)}$ is a strongly invariant collection (see Definition  \ref{i-ii}).  
We claim that
\begin{equation}\label{d2} 
d\big(f^n(y), z\big)> \frac{\varepsilon}{2},  \quad \forall y\in Z(f),\,\,
\forall z\in S(f), \,\, \forall n\ge n_1.
\end{equation}
In fact, given $y\in Z(f)$ and $n\ge n_1$, there exists $\alpha=(i_0, i_1,\ldots, i_{n-1})\in\mathcal{I}_n(f)$ such that
$f^n(y) = \varphi^\alpha (y)$. Then, 
$$ d\big(f^n(y), \varphi^\alpha (x_0) \big) = d\big(\varphi^\alpha(y), \varphi^\alpha(x_0) \big) \le \textrm{diam}(X)\lambda^n< \dfrac{\varepsilon}{2}, \quad \forall\, n\ge n_1.
$$
The claim now follows by using the triangle inequality in \eqref{lem:hypothesis}.

We start by verifying that $\mathscr{C}^{(n_1)}$ satisfies item (i) of Definition \ref{i-ii}. 
Let $\alpha=(i_0, i_1, \ldots, i_{n_1-1})\in \mathcal{I}_{n_1}(f)$ and $A^{\alpha}\in \mathscr{C}^{(n_1)}$. By \eqref{lem:claim 2},  we have
$$ 
f\big(A^{\alpha}\big)\subset  \bigcup_{ j:(i_1, \ldots, i_{n_1-1}, j)\in \mathcal{I}_{n_1}(f)} A^{(i_1, \ldots, i_{n_1-1}, j)}. 
$$
We claim that there is $i_{n_1}\in\mathcal{A}$ such that
\begin{equation}\label{lem: item i}
f\big(A^{\alpha}\big)\subset A^{\alpha'}\,,
\end{equation} where $\alpha'=(i_1, \ldots, i_{n_1-1}, i_{n_1})\in \mathcal{I}_{n_1}(f)$.

To prove the claim, we argue by contradiction: let $\alpha=(i_0,i_1,\ldots,i_{n_1-1})\in \mathcal{I}_{n_1}(f)$ and suppose that there are $x,y \in A^{\alpha}$ and
 $j, j' \in \mathcal{A}$ with $j\neq j'$ such that  $f(x)\in A^{(i_1,\ldots,i_{n_1-1}, j)}$ and $f(y)\in A^{(i_1,\ldots, i_{n_1-1}, j')}$. In particular, $f^{n_1}(x)\in A_{j}$ and $f^{n_1}(y)\in A_{j'}$. Then, 
$$
d(f^{n_1}(x),f^{n_1}(y))=d(\varphi^\alpha(x),\varphi^\alpha(y))\leq \text{diam}(X)\lambda^{n_1}< \frac{\varepsilon}{2},
$$
which means that both $f^{n_1}(x)$ and $f^{n_1}(y)$ belong to an open ball $B_{\frac{\varepsilon}{2}}$ of radius $\varepsilon/2$. Because in the metric space $(X,d)$, open balls are connected, we have that $B_{\frac{\varepsilon}{2}}$ is a connected set that intersects the pairwise disjoint connected sets $A_j$ and $A_{j'}$. In this way, $B_{\frac{\epsilon}{2}}$ must contain a point of $S(f)$. But this shows that $f^{n_1}(x)$ is $\varepsilon/2$-close to $S(f)$,  which contradicts \eqref{d2}.
 Hence, the claim is true. Therefore,  item (i) of Definition \ref{i-ii} holds.

We will now verify item (ii) of Definition \ref{i-ii}. Let $\alpha\in \mathcal{I}_{n_1}(f)$ and $n\ge n_1$.  By applying \eqref{lem: item i} finitely many times, we have $f^{n}\big(A^{\alpha}\big)\subset A^{\beta}$ for some $\beta\in \mathcal{I}_{n_1}(f)$.  Moreover, by \eqref{d2},  we conclude that 
$\overline{f^{n}\big(A^{\alpha}\big)}\subset A^{\beta}$.
This shows that item (ii) of Definition \ref{i-ii} holds for every $n\geq n_1$.

Finally,  item (iii) of Definition \ref{i-ii} is automatic from the definition of $\mathscr{C}^{(n)}$ given in \eqref{Qna}.

 \end{proof}

Thus, we obtain a sufficient condition for $f$ to be asymptotically periodic on $Z(f)$. 

\begin{lem}\label{lem:suff} If $\Omega(f)\cap S(f)=\emptyset$, then $f$ is asymptotically periodic on $Z(f)$.
\end{lem}
\begin{proof}
It follows from Lemma~\ref{lem333} and Lemma~\ref{lem:suff1}.
\end{proof}

\subsection{Proof of Theorem~\ref{th:main}}\label{proof:main}
In what follows, assume that  $(X,d)$ is a compact metric space whose open balls are connected and   $\lambda \in (0,1)$. Suppose that $U\subset\mathbb{R}^M$ is a set of positive Lebesgue measure and $\{f_{\mu}\}_{\mu\in U}$ is a family of $\lambda$-piecewise contractions on $X$
satisfying the hypotheses (E) and (T) at $\mu^*\in U$.  Recall that $S_\mu=S(f_\mu)$ is the singular set of $f_\mu$ and $Z_\mu=Z(f_\mu)$ is the set of regular points of $f_\mu$ (see \eqref{def Z}). Let $x_0\in X$, $\varepsilon_0>0$, $\delta_0>0$, $n_0\ge 1$, $a>0$ and $c>0$ be as in Definition \ref{transversality}.
Let
$$
V:=\{\mu\in U \colon f_\mu\text{ is not asymptotically periodic on }Z_\mu\}.
$$
We denote by $\{\varphi_{i,\mu}\}$ the iterated function system associated to each piecewise $\lambda$-contraction $f_{\mu}$, as in Definition \ref{def1}. Given an $n$-tuple $\alpha=(i_0,i_1,\ldots,i_{n-1})\in\mathcal{A}^n$, we let $\varphi^{\alpha}_\mu$ denote the map defined in \eqref{def varphi}.
 Replacing $f$ by $f_{\mu}$ and $\varphi^{\alpha}$ by
$\varphi^{\alpha}_\mu$ in \eqref{omegaf}, we obtain the definition of $\Omega_\mu:=\Omega(f_{\mu})$. 

The claim of Theorem \ref{th:main}  is trivially true if $X$ consists of a single point. Hence, we may assume that 
$X$ has at least two points. In this way,  $\mathrm{diam}(X)>0$ and all the lemmas proved in this section hold.

Let $\delta>0$ be arbitrary. By Lemma~\ref{lem:suff} and by the definition of $U_{\delta}(\mu^*)$ given in \eqref{J},
\begin{equation}\label{A1}
V\cap U_{\delta}(\mu^*)\subset \{\mu\in U_{\delta}(\mu^*)\colon \Omega_\mu \cap S_\mu\neq\emptyset\}.
\end{equation}
By Lemma~\ref{lem:covering},  for each $n\ge 1$ and for each $\mu\in U_{\delta}(\mu^*)$, we can cover $\Omega_\mu$ using at most $\# \mathcal{J}_n^\delta(\mu^*)$ closed balls of radius  $\varepsilon_n:=2\text{diam}(X) \lambda^n$ centred at the points $\varphi^{\alpha}_\mu(x_0)$ with $\alpha\in\mathcal{I}_n(f_\mu)\subset \mathcal{J}_{n}^\delta(\mu^*)$. Hence, by \eqref{A1}, we obtain
\begin{equation}\label{A2}
\text{Leb}^*(V\cap U_{\delta}(\mu^*))\leq \sum_{\alpha\in \mathcal{J}_n^\delta(\mu^*)}\text{Leb}^*\big(\{\mu\in U_{\delta}(\mu^*)\colon d(\varphi_\mu^\alpha(x_0),S_{\mu})\leq \varepsilon_n\}\big).
\end{equation}
Let $n_1$ be the least positive integer such that $\varepsilon_n<\varepsilon_0$ for all $n\ge n_1$.  By Hypothesis (T), for each $n\ge \max\, \{n_0 ,n_1\}$, 
\begin{equation}\label{A3}
\text{Leb}^*\big(\{\mu\in U_{\delta_0}(\mu^*)\colon d(\varphi_\mu^\alpha(x_0),S_{\mu})\leq \varepsilon_n\}\big)\leq c\, \varepsilon_n^a,\quad \forall\, \alpha\in\mathcal{J}_n^{\delta_0}(\mu^*).
\end{equation}
Moreover, $0<\delta\le \delta_0$ implies $U_{\delta}(\mu^*)\subset U_{\delta_0}(\mu^*)$ and $\mathcal{J}_n^{\delta}(\mu^*)\subset \mathcal{J}_n^{\delta_0}(\mu^*)$ so that \eqref{A3} reads
\begin{equation}\label{A4}
\text{Leb}^*\big(\{\mu\in U_{\delta}(\mu^*)\colon d(\varphi_\mu^\alpha(x_0),S_{\mu})\leq \varepsilon_n\}\big)\leq c\, \varepsilon_n^a,\quad \forall\, \alpha\in \mathcal{J}_n^{\delta}(\mu^*).
\end{equation}
By \eqref{A2}  and \eqref{A4}, for all $0<\delta\le \delta_0$, $n\ge \max\,\{n_0, n_1\}$ and $\alpha\in\mathcal{J}_n^{\delta}(\mu^*)$,
$$
\text{Leb}^*(V\cap U_{\delta}(\mu^*))\leq c\,\varepsilon_n^a\#\mathcal{J}_n^\delta(\mu^*).$$
 Therefore, making $n\to\infty$ and using the definition of $\varepsilon_n$ yields
\begin{equation}\label{A5}
\text{Leb}^*(V\cap U_{\delta}(\mu^*))\leq c\,(2\text{diam}(X))^a\limsup_{n\to\infty}\lambda^{an}\#\mathcal{J}_n^\delta(\mu^*).
\end{equation}
for all $0<\delta \le \delta_0$.

Since $0< \lambda<1$, there exists $b>0$ so small that $\lambda^a e^{a b}<1$. By Hypothesis (E), there exist $0<\delta<\delta_0$ and $n_2\in\mathbb{N}$ such that
\begin{equation}\label{A6}
 \lambda^{an} \#\mathcal{J}_n^{\delta}(\mu^*) <  \lambda^{an} e^{ab n}=\big(\lambda^a e^{ab}\big)^n,\quad\forall n\ge n_2.
\end{equation}
By \eqref{A5} and \eqref{A6}, we obtain    $\textrm{Leb}^* \big(V\cap U_{\delta}(\mu^*)\big)=0$ for some $0<\delta<\delta_0$. This concludes the proof of Theorem~\ref{th:main}.
\qed

\subsection{Sufficient condition for Hypothesis (E)}

Let $\{f_\mu\}_{\mu\in U}$ denote a family of piecewise $\lambda$-contractions on $X$. In what follows, we establish a sufficient condition that guarantees that $\{f_\mu\}_{\mu\in U}$ satisfies Hypothesis (E) at $\mu^*\in U$. 

First we  introduce the notion of multiplicity following \cite{MR1462857}.  Given a finite collection $\mathscr{C}$ of subsets of $X$, the \textit{multiplicity} of $\mathscr{C}$ is the number
$$
\text{mult}(\mathscr{C}) :=\max_{x\in X }\#\{A\in\mathscr{C}\colon x\in \overline{A}\}.
$$
Given $m\in\Nn$, a real number $r>0$ is \textit{$m$-compatible} with $\mathscr{C}$ if every open ball with radius $r$ meets at most $m$ elements of $\mathscr{C}$. If $r$ is $\text{mult}(\mathscr{C})$-compatible with $\mathscr{C}$, then we simply say that $r$ is \textit{compatible} with $\mathscr{C}$. Denote by $B_r(x)$ the open ball of radius $r>0$ centred at $x\in X$.

\begin{lem}\label{positive tau} There exists $r>0$ such that $r$ is compatible with $\mathscr{C}$.
\end{lem}
\begin{proof}
Suppose, by contradiction, that for every $n\geq1$ there is $x_n\in X$ such that the open ball $B_{1/n}(x_n)$ intersects at least $m:=\text{mult}(\mathscr{C})+1$ elements of $\mathscr{C}$, say $A_1^{(n)},\ldots,A_m^{(n)}\in \mathscr{C}$. Since $\mathscr{C}$ is a finite collection, by the pigeonhole principle,  
there exist $A_1,\ldots, A_m\in\mathscr{C}$ and $n_k\to\infty$ as $k\to\infty$
 such that  $A_1^{(n_k)}=A_1,\ldots,A_m^{(n_k)} =A_m$ for all $k\ge 1$. Since $X$ is compact, by taking a subsequence of $\{n_k\}$ if necessary, we may assume that  there exists $x\in X$ such that $x_{n_k}\to x$ as $k\to\infty$. For $ i=1,\ldots, m$ and $k\ge 1$, let $y_i^{(n_k)} \in A_{i}\cap B_{\frac{1}{n_k}}\big(x_{n_k}\big)$.
 By the triangle inequality, we have that $y_i^{(n_k)}\to x$ as $k\to\infty$, for $i=1,\ldots,m$. Thus, $x\in \overline{A_i}$ for $i=1,\ldots,m$, which contradicts the fact that $\#\{A\in\mathscr{C}\colon x\in\overline{A}\}\leq \text{mult}(\mathscr{C})$. 
\end{proof}

Given $m\in\Nn$, define
$$
\tau_m(\mathscr{C}) :=\sup\left\{r>0\colon r \text{ is $m$-compatible with }\mathscr{C}\right\},
$$
and $\tau(\mathscr{C})=\tau_{\text{mult}(\mathscr{C})}(\mathscr{C})$. Notice that $\tau_m(\mathscr{C})=+\infty$ if $m\geq \# \mathscr{C}$. By Lemma~\ref{positive tau}, we have $\tau_m(\mathscr{C})>0$ whenever $m\geq\text{mult}(\mathscr{C})$.

Denote by $\mathcal{A}=\{1,\ldots, N\}$ the label set of $\{f_\mu\}_{\mu\in U}$ and by $\mathscr{C}_\mu^{(1)}=\{A_{i,\mu}\}_{i\in\mathcal{A}}$, $\mu\in U$,  the partition of $f_\mu$ as in Definition \ref{def2}. 
Let $\mathscr{C}^{(n)}_\mu$ with $n\geq 2$ be the collection of non-empty open sets 
\begin{equation}\label{Qna2}
A_\mu^{\alpha}:=A_{i_0, \mu}\cap\varphi_{i_0,\mu}^{-1}(A_{i_1,\mu})\cap\cdots\cap (\varphi_{i_{n-2},\mu}\circ\cdots\circ\varphi_{i_0,\mu})^{-1}(A_{i_{n-1}, \mu} ), 
\end{equation}
where $\alpha=(i_0,i_1,\ldots,i_{n-1})$  runs over the set $\mathcal{A}^n$.

In what follows, given two subsets $A,B$ of the metric space $(X,d)$, we denote by $d_H(A,B)$ the Hausdorff distance of $A$ and $B$ defined as
$$ 
d_H(A,B) = \max \left\{\sup_{a\in A} d(a,B), \sup_{b\in B} d(b, A) \right\}, 
$$
where $d(x,A) = \inf_{a\in A} d(x,a)$, $x\in X$.

\begin{defn}\label{stablep}
Given $n\geq 1$, we say that $\big\{\mathscr{C}^{(n)}_\mu\big\}_{\mu\in U}$ is \textit{stable} at $\mu^*\in U$ if there is $\delta>0$ such that  the following conditions are satisfied:
\begin{itemize}
\item [$(i) $] $\mathcal{I}_n(f_\mu)=\mathcal{I}_n(f_{\mu^*})$ for every  $\mu\in U_\delta(\mu^*)$;
\item [$(ii) $] $d_H\left( A^{\alpha}_{\mu},  A^{\alpha}_{\mu^*} \right)\to 0$  as $\mu\to \mu^*$,  for every $\alpha\in \mathcal{I}_n(f_{\mu^*})$.
\end{itemize}
\end{defn}

As in \cite{MR1462857}, we introduce the multiplicity entropy.

\begin{defn}\label{def:mult entropy}
The \textit{multiplicity entropy} of $f_\mu$ is
$$
H_{\mathrm{mult}}(f_\mu)=\limsup_{n\to\infty}\frac1n\log \mathrm{mult}(\mathscr{C}^{(n)}_{\mu}).
$$
\end{defn}
Informally, the multiplicity entropy measures the exponential growth rate of the number of itineraries sharing a common prefix. 

The following result establishes a sufficient condition for Hypothesis (E) that is easier to check. 
 
\begin{prop}\label{lem26}
If $\big\{\mathscr{C}^{(n)}_\mu\big\}_{\mu\in U}$ is stable at $\mu^*\in U$ for every $n\geq 1$ and $H_{\mathrm{mult}}(f_{\mu^*})=0$, 
then $\{f_\mu\}_{\mu\in U}$ satisfies Hypothesis (E) at $\mu^*$.
\end{prop}

\begin{proof}
Given $\varepsilon>0$,  choose $p\in\Nn$ sufficiently large so that $ m:=\mathrm{mult}(\mathscr{C}^{(p)}_{\mu^*}) <e^{\varepsilon p}$.  We start by proving two claims.

\medskip

\noindent Claim A.  There is $\delta_0=\delta_0(p)>0$ such that  $\mathrm{mult}(\mathscr{C}^{(p)}_{\mu})\leq m$ for every $\mu\in U_{\delta_0}(\mu^*)$.  

\medskip

Suppose by contradiction that there is a sequence $(\mu_k)_{k\geq1}$ converging to $\mu^*$ such that $\mathrm{mult}(\mathscr{C}^{(p)}_{\mu_k})>m$.  Hence, for each $k\ge 1$, there is $x_k\in X$ such that $\#\{A\in\mathscr{C}^{(p)}_{\mu_k}\colon x_k\in \overline{A}\}>m$.  This implies that there are $\alpha_{k,1},\ldots,\alpha_{k,m+1}\in \mathcal{I}_p(f_{\mu_k})$ such that $x_k\in\bigcap_{i=1}^{m+1}\overline{A_{\mu_k}^{\alpha_{k,i}}}$.  Since $\big\{\mathscr{C}_{\mu}^{(p)}\big\}_{\mu\in U}$ is stable at $\mu^*$,  by item $(i)$ of \mbox{Definition \ref{stablep}},
$\mathcal{I}_p(f_{\mu_k})= \mathcal{I}_p(f_{\mu^*})$ for all $k\in\mathbb{N}$  big enough.  Therefore,  by the compactness of $X$ and the pigeonhole principle,  there exist an infinite subset $\mathbb{N}'\subset\mathbb{N}$, $x\in X$  and $\alpha_{1},\ldots,\alpha_{m+1}\in \mathcal{I}_p(f_{\mu^*})$ such that $x_k\to x$ as $k\to\infty$ in $\mathbb{N}'$ and $x_{k}\in\bigcap_{i=1}^{m+1}\overline{A_{\mu_{k}}^{\alpha_{i}}}$ for every $k\in\mathbb{N}'$.  We affirm that  $x\in \bigcap_{i=1}^{m+1}\overline{A_{\mu^*}^{\alpha_{i}}}$. In fact, 
\begin{equation}\label{a44}
 d(x, A_{\mu^*}^{\alpha_i}) \le d(x, x_k) + d\big(x_k,  {A_{\mu_k}^{\alpha_i}}\big) + d_H\big(  {A_{\mu_k}^{\alpha_i}}, A_{\mu^*}^{\alpha_i}\big), \,\, \forall k\in\mathbb{N}, \,\, \forall i\in \{1,\ldots, m+1\}
\end{equation}
Since $x_{k}\in \overline{A_{\mu_{k}}^{\alpha_{i}}}$, we have that $d\big(x_k, {A_{\mu_k}^{\alpha_i}}\big) =0$. By \eqref{a44}, by  item $(ii)$ of \mbox{Definition \ref{stablep}}  and by the fact that $d(x,x_k)\to 0$ as $k\to\infty$ in $\mathbb{N}'$, we conclude that $d\big(x,A_{\mu^*}^{\alpha_i}\big)=0$, that is, $x\in \overline{A_{\mu^*}^{\alpha_i} }$. Because $i$ is arbitrary, we have that $x\in \bigcap_{i=1}^{m+1}\overline{A_{\mu^*}^{\alpha_{i}}}$, which implies that $\mathrm{mult}(\mathscr{C}^{(p)}_{\mu^*})\geq \#\{A\in\mathscr{C}^{(p)}_{\mu^*}\colon x\in \overline{A}\}>m$, yielding a contradiction and proving Claim A.

\medskip 

\noindent Claim B. Let $\delta_0$ be as in Claim A, then
$
\displaystyle r:=\inf_{\mu\in U_{\delta_0}(\mu^*)}\tau_m(\mathscr{C}^{(p)}_{\mu})
>0.$

\medskip 

By Lemma~\ref{positive tau} and Claim A,  we have $\tau_m(\mathscr{C}^{(p)}_{\mu})>0$ for every $\mu\in U_{\delta_0}(\mu^*)$.  Now,  
suppose by contradiction  that there is a sequence $(\mu_k)_{k\geq1}$ converging to $\mu^*$ such that $\tau_m(\mathscr{C}^{(p)}_{\mu_k})\to 0$ as $k\to\infty$. So, for some $k_0\ge 1$ and for each $k\geq k_0$,  there exist  $x_k\in X$ and $r_k>\tau_m(\mathscr{C}^{(p)}_{\mu_k})$ such that $r_k\to 0$ as $k\to\infty$ and $\#\{A\in \mathscr{C}^{(p)}_{\mu_k}\colon B_{r_k}(x_k)\cap A\neq \emptyset\}>m$. Proceeding as in in the proof of Claim A, we can show that there exist an infinite subset $\mathbb{N}'\subset\mathbb{N}$, $x\in X$  and $\alpha_{1},\ldots,\alpha_{m+1}\in \mathcal{I}_p(f_{\mu^*})$ such that $x_k\to x$ as $k\to\infty$ in $\mathbb{N}'$ and $d\big(x_k, A_{\mu_k}^{\alpha_i}\big)<r_k$ for every $k\in\mathbb{N}'$ and $i\in \{1,\ldots,m+1\}$. By \eqref{a44}, we conclude that $x\in \bigcap_{i=1}^{m+1}\overline{A_{\mu^*}^{\alpha_{i}}}$, which implies that $\mathrm{mult}(\mathscr{C}^{(p)}_{\mu^*})\geq \#\{A\in\mathscr{C}^{(p)}_{\mu^*}\colon x\in \overline{A}\}>m$, yielding a contradiction and proving Claim B. 

\medskip 

Now,  we conclude the proof of the proposition using Claim B.  For each  
 $n\ge 1$, $\mu\in U$ and  $\alpha\in \mathcal{I}_n(f_\mu)$, since $f_{\mu}^n$ is a Lipschitz $\lambda^n$-contraction on $A_\mu^{\alpha}\in \mathscr{C}^{(n)}_\mu$, we have that, 
$$
\textrm{diam}\left(f_\mu^n\big(A_\mu^{\alpha} \big) \right)\le \lambda^n \textrm{diam}\left(A_\mu^{\alpha}\right)\le \lambda^n \textrm{diam}(X).
$$
By Claim B, choosing $n_0\geq 1$ so that $\rho:=\lambda^{n_0} \textrm{diam}(X)< r$, we get 
$$\textrm{diam}(f_\mu^n\big(A_\mu^{\alpha} \big) )\leq \rho< r\leq \tau_{m}(\mathscr{C}_\mu^{(p)}), \quad \forall n\geq n_0, \quad \forall \mu\in U_{\delta_0}(\mu^*). $$   Hence, for all $n\ge n_0$, we obtain
  $$ \# \left\{\beta\in \mathcal{I}_{p}(f_\mu): f_\mu^n (A_\mu^{\alpha})\cap A_\mu^{\beta}\neq\emptyset \right\}\le m,\quad \forall\,n\geq n_0, \quad \forall\mu\in U_{\delta_0}(\mu^*).
  $$
  In this way, for each $n\ge n_0$  and $\mu\in U_{\delta_0}(\mu^*)$,
  $$ \#\mathcal{I}_{n+p}(f_\mu) \le m \#\mathcal{I}_n(f_{\mu}).
  $$
 Therefore,
  $$ \#\mathcal{J}_{n+p}^{\delta}(\mu^*) \le m\# \mathcal{J}_n^{\delta}(\mu^*),\quad \forall \delta\in [0,\delta_0), \quad\forall n\ge n_0. 
  $$
In particular, we have that
$$ 
 \#\added{\mathcal{J}}_{n_0+ip}^\delta(\mu^*) \le m^i \#\added{\mathcal{J}}_{n_0}^\delta(\mu^*), ,\quad \forall \delta\in [0,\delta_0), \quad\forall i\ge 0.
 $$
Hence,  we obtain
  $$  \#\mathcal{J}_{n}^{\delta}(\mu^*) \le m^{\frac{n-n_0}{p}+1}\# \mathcal{J}_{n_0}^{\delta}(\mu^*), \quad \forall \delta\in [0,\delta_0), \quad\forall n\ge n_0. 
  $$  
By our choice of $p$, for $C:=\log\left(m^{(p-n_0)/p}\# \mathcal{J}_{n_0}^{\delta_0}(\mu^*)\right)$, we have that
$$\log  \#\mathcal{J}_{n}^{\delta}(\mu^*) \le n\varepsilon+C, \quad \forall \delta\in [0,\delta_0), \quad\forall n\ge n_0. 
$$
Thus,
$$ \limsup_{n\to\infty} \frac1n \log\#\mathcal{J}_n^\delta(\mu^*) \le \varepsilon, \quad \forall \delta\in [0,\delta_0).
$$
Since $\varepsilon>0$ is arbitrary, we conclude that hypothesis (E) holds true at $\mu^*$.

\end{proof}

\section{One-dimensional piecewise contractions}\label{onedimensionalpcs}

In order to illustrate the applicability of Theorem \ref{th:main}, we will use it to give an alternative proof of the almost surely asymptotic periodicity of one-dimensional piecewise contractions with varying partition originally proved in \cite[Theorem 1.4]{NPR18}. 

Throughout this section, let $N\geq2, \,\, 0<\lambda<1$ and $\varphi_1,\ldots,\varphi_N:X\to X$ be bi-Lipschitz $\lambda$-contractions on $X=[0,1]$.  Consider  the open set $$U=\{(\mu_1,\ldots,\mu_{N-1})\in\mathbb{R}^{N-1}: 0< \mu_1<\cdots<\mu_{N-1}<1\}. $$
 For each $\mu=(\mu_1,\ldots,\mu_{N-1})\in U$, let $A_{i,\mu}=(\mu_{i-1},\mu_i)$ for $i=1,\ldots,N$. where we set $\mu_0=0$ and $\mu_N=1$, and let $f_\mu:[0,1]\to [0,1]$ be a piecewise $\lambda$-contraction with label set $\mathcal{A}=\{1,\ldots, N\}$,  partition set $\{A_{i,\mu}\}_{i\in\mathcal{A}}$ and  iterated function system $\{\varphi_i\}_{i\in\mathcal{A}}$. We denote by $S_
 \mu=S(f_\mu)$ the singular set of $f_\mu$. 
%\begin{equation}\label{fmu50}
% f_\mu (x) = \varphi_i (x) \quad\textrm{if}\quad x\in \big(\mu_{i-1}, \mu_i\big),
%\end{equation}

We make two assumptions:
\begin{enumerate} 
\item[(A1)] For each $\mu\in U$, $f_{\mu}$ is either left or right continuous at each point of the singular set $S_{\mu}$;
\item[(A2)] $\varphi_i([0,1])\subset (0,1)$ for every $i\in\mathcal{A}$. 
\end{enumerate}

The assumption (A2) is  convenient and not restrictive,  since we can always extend the domain of the bi-Lipschitz contractions $\varphi_i$'s to a larger interval,  having the same dynamics of $f_\mu$ in $X$ and satisfying the desired assumption for the iterated function system.

Based on Theorem \ref{th:main}, we will proof the following result.

\begin{thm}[{Nogueira-Pires-Rosales (2018)}] \label{th:NPR2}
Let $\{f_\mu\}_{\mu\in U}$ be a family of piecewise $\lambda$-contractions on $[0,1]$ satisfying assumptions (A1) and (A2). For Lebesgue almost every $\mu\in U$, the map $f_{\mu}$  is asymptotically periodic and the periodic orbits are regular.
\end{thm}

We now proceed to prove Theorem~\ref{th:NPR2}.

\subsection{Proof of Theorem~\ref{th:NPR2}}
Denote by $\Phi=\{\varphi_i\}_{i\in\mathcal{A}}$ the iterated function system of $f_{\mu}$ and by $\varphi^{\alpha}$ the map defined in \eqref{4344}.  Notice that the maps in the iterate function system do not depend on $\mu$. According to Definition~\ref{def2}, $\{f_\mu\}_{\mu\in U}$  is a family of piecewise $\lambda$-contractions of the interval $[0,1]$ sharing
the same label set $\mathcal{A}$ and the same iterated function system $\Phi$. However, the partition $\{A_{i,\mu}\}_{i\in\mathcal{A}}$ of $f_\mu$ varies with the parameter $\mu\in U$. Notice that the singular set of $f_{\mu}$ is
$$
S_{\mu}=\{0,\mu_1,\mu_2,\ldots,\mu_{N-1},1\}.
$$
\begin{defn}
Given $\mu\in U$, we say that the pair $(\Phi,\mu)$ has a \textit{singular connection} if there are $n\geq 1$ and $\alpha\in\mathcal{A}^n$ such that $\varphi^\alpha(\mu_i)=\mu_j$ for some $(i,j)\in\{1,\ldots,N-1\}^2$. 
\end{defn}
\begin{lem}\label{lem:sing connections}
For Lebesgue almost every $\mu\in U$, the pair $(\Phi,\mu)$ has no singular connection.
\end{lem}

\begin{proof}
 Given $n\geq1$,  $\alpha\in\mathcal{A}^n$ and  $(i,j)\in\{1,\ldots,N-1\}^2$, consider the set defined by
$$
B^{\alpha}_{ij}=\left\{(\mu_1,\ldots,\mu_{N-1})\in U\colon \varphi^\alpha(\mu_i)=\mu_j  \right\}.
$$
We claim that $B^{\alpha}_{ij}$ is a Lebesgue null set. Indeed, taking into account that $\varphi^\alpha$ is a Lipschitz contraction, we have two cases:
\begin{enumerate}
\item In the first case, $i=j$, the equality $\varphi^\alpha(\mu_i)=\mu_i$ implies that $\mu_i$ equals the unique fixed point of $\varphi^\alpha$. Thus, $B_{ii}^\alpha$ is a null set. 
\item In the second case, $i\neq j$,  being the graph of a Lipschitz function,  $B_{ij}^\alpha$ is a null set.
\end{enumerate}
Therefore, 
$$
B:=\bigcup_{n\geq1}\bigcup_{\alpha\in \mathcal{A}^n} \bigcup_{i=1}^{N-1}  \bigcup_{j=1}^{N-1} B^{\alpha}_{ij}
$$
is  a null set. This means that for each $\mu\in U{\setminus} B$, the pair $(\Phi,\mu)$ has no singular connection. 
%By a theorem of Wells \cite{W75}, the set $\{y\in [0,1]\colon (\varphi^\alpha)^{-1}(y)\text{ is infinite}\}$ has zero Lebesgue measure. 
%The result follows.
\end{proof}
Let

$$
U^*=\left\{\mu\in U\colon \text{$(\Phi,\mu)$ has no singular connection}\right\}.
$$

\begin{lem}\label{ppar}
The periodic points of $f_{\mu}$ are regular for every $\mu\in U^*$.
\end{lem}
\begin{proof}
Suppose by contradiction that $x\in [0,1]$ is a non-regular periodic point of $f_{\mu}$ with $\mu=(\mu_1,\ldots,\mu_{N-1})\in U^*$.  By assumption (A2),  we conclude that the periodic orbit of $x$ does not intersect $\{0,1\}$.  Therefore,  there must exist $n\geq 0$ and $i\in\{1,\ldots,N-1\}$ such that $f^n_{\mu}(x)=\mu_i$.  But $x$ is periodic, say with period $p$,  thus $\mu_i=f^{n}_{\mu}(x)=f^{n+p}_{\mu}(x)=f^p_{\mu}(\mu_i)$.  By assumption (A1),  there is a $p$-tuple $\alpha\in\mathcal{A}^p$ such that $\varphi^\alpha(\mu_i)=f_{\mu}^p(\mu_i)=\mu_i$,  which contradicts the fact that $(\Phi,\mu)$ has no singular connection.  
\end{proof}

Hereafter, fix $\mu^*\in U^*$.

\begin{lem}\label{lem:34}
$\{f_\mu\}_{\mu\in U}$ satisfies Hypothesis (E) at $\mu=\mu^*$.
\end{lem}

\begin{proof}
Let $n\geq 1$ and recall from  \eqref{4344} that 
$
\mathscr{C}_\mu^{(n)} =\{ A_\mu^{\alpha}\colon \alpha \in \mathcal{I}_n(f_\mu)\}
$.
Because the functions defining the iterative function system are bi-Lipschitz, $\mathscr{C}_\mu^{(n)}$ is a finite collection of open intervals whose union equals $[0,1]$ but finitely many points  that are not regular points of order $n$ of $f_\mu$. This implies that  $\textrm{mult}\big(\mathscr{C}_\mu^{(n)}\big)\leq 2$ for every $\mu\in U$.
By Proposition~\ref{lem26}, the proof will be finished provided we show that $\{\mathscr{C}_\mu^{(n)}\}_{\mu\in U}$ is stable at $\mu=\mu^*$, i.e., that conditions (i) and (ii) in Definition \ref{stablep} are true. We begin by verifying condition (i) in Definition \ref{stablep}. The next results are true for all $n\ge 1$.\\

\noindent  Claim A.  $\mathcal{I}_n(f_{\mu^*})\subset \mathcal{I}_n(f_\mu)$ for some $\delta>0$ and all $\mu\in U_{\delta}(\mu^*)$.\\

 In fact, let $\alpha= (i_0, i_1,\ldots, i_{n-1})\in \mathcal{I}_n(f_{\mu^*})$. Then there is a regular point $x\in (0,1)$ of order $n$ such that
$f_{\mu^*}^k(x)\in A_{i_k,\mu}$ for all $0\le k\le n-1$. Let $\delta_\alpha:=\min_{0\le k\le n-1} d\big(f_{\mu^*}^k(x), S_{\mu^*} \big)>0$. We have that
$f_{\mu}^k(x) = f_{\mu^*}^k(x)$ for all $\mu\in U_{\delta_{\alpha}}(\mu^*)$ and $0\le k\le n-1$. In particular, $x$ is a regular point of $f_\mu$ of order $n$ whose itinerary of order $n$ equals $\alpha$, that is, $\alpha\in \mathcal{I}_n(f_\mu)$ for all $\mu\in U_{\delta_{\alpha}}(\mu^*)$. Since $\alpha\in\mathcal{I}_n(f_\mu^*)$ is arbitrary in a finite set, we have that the claim is true for $\delta=\min_{\alpha\in\mathcal{I}_n(f_\mu^*)}\delta_\alpha>0$.  \\

\noindent Claim B. $\mathcal{I}_n(f_\mu) = \mathcal{I}_n(f{_\mu^*})$ for some $\delta>0$ and all $\mu\in U_{\delta}(\mu^*)$. \\

By Claim A, it  suffices showing that
$\#\mathcal{I}_n(f_{\mu^*})= \#\mathcal{I}_n(f_\mu)$. Notice that  $\#\mathcal{I}_n(f_\mu)$ equals the number of members of $\mathscr{C}_\mu^{(n)}$, or equivalently,  the number of connected components of $(0,1){\setminus} D_{\mu}^{(n)}$, where 
$$D_\mu^{(1)} = S_\mu\quad\textrm{and}\quad  D_{\mu}^{(n)} = \bigcup_{k=0}^{n-1} f_\mu^{-k}(S_\mu).$$
In this way, the proof of the claim will be finished if we show that $\# D_{\mu}^{(n)} = \#D_{\mu^*}^{(n)}$. By (A2), we have that
\begin{equation}\label{d19}
D_\mu^{(n)} = \bigcup_{i=1}^{N-1} \bigcup_{k=0}^{n-1} f_\mu^{-k}(\mu_i).
\end{equation}    
Since $(\Phi,\mu^*)$ has no singular connection, for each $1\le i\le N-1$ and $1\le k\le n-1$, there exist a finite set
$F_{i, k}$ (possibly empty) and $\delta_{i,k}>0$ such that
\begin{equation}\label{d20}
 f_{\mu}^{-k} (\mu_i) = \bigcup_{j\in F_{i,k}} \psi_{j}^{-1}(\mu_i), \quad \forall \mu\in U_{\delta_{i,k}}(\mu^*),
\end{equation} 
where $\psi_j$ is the  composition of finitely many maps of $\Phi=\{\varphi_i\}_{i\in\mathcal{A}}$. Since each $\varphi_i$ is bi-Lipschitz, it follows from \eqref{d19} and \eqref{d20} that 
$\#D_{\mu}^{(n)}=\# D_{\mu^*}^{(n)}$ for all $\mu\in U_{\delta}(\mu^*)$, where  $\delta= \min_{i, k} \delta_{i,k}$, proving Claim B.  \\

Moreover, there exists a positive constant $M_{i,k,j}$ depending on $i,k,j$ such that for all $\mu\in U_{\delta}(\mu^*)$,
\begin{eqnarray*}
d_H\left( D_{\mu}^{(n)}, D_{\mu^*}^{(n)} \right) &=& 
d_H\left(\bigcup_{i=1}^{N-1} \bigcup_{k=0}^{n-1}  \bigcup_{j\in F_{i,k}} \psi_{j}^{-1}(\mu_i), \bigcup_{i=1}^{N-1} \bigcup_{k=0}^{n-1} \bigcup_{j\in F_{i,k}} \psi_{j}^{-1}(\mu_i^*), \right) \\ &\le&
\max_{1\le i\le N-1} \max_{0\le k\le n-1}  \max_{j\in F_{i,k}}  M_{i,k,j} \left\vert \mu_i -\mu_i^*\right\vert \\
&\le& M \Vert \mu - \mu^*\Vert,
\end{eqnarray*}
where $\delta= \min_{i, k} \delta_{i,k}$ and $M= \max_{i} \max_{k}  \max_{j\in F_{i,k}}  M_{i,k,j}$.

Since the members of $\mathscr{C}_\mu^{(n)}=\{A_\mu^{\alpha}: \alpha\in \mathcal{I}_n(f_\mu)\}$ are the connected components of $(0,1){\setminus} D_\mu^{(n)}$, we have proved the following claim.\\

\noindent Claim C. $d_H\big(A_{\mu}^\alpha, A_{\mu^*}^\alpha \big)\to 0$  as $\mu\to\mu^*$ in $U_{\delta}(\mu^*)$  for every $\alpha\in\mathcal{I}_n(f_{\mu^*})$.\\

In other words, condition (ii) in Definition \ref{stablep} is true.
\end{proof}

\begin{lem}
$\{f_\mu\}_{\mu\in U}$ satisfies Hypothesis (T) at $\mu=\mu^*$.
\end{lem}

\begin{proof}
It is clear that for every $\delta_0>0$, $\varepsilon>0$ and  $n\geq 1$, we have
$$
\text{Leb}^*\,\big(\{\mu\in U \colon \min_{i}| \varphi^\alpha(0)-\mu_i|\leq \varepsilon\} \big)\leq 2(N-1)\varepsilon,\quad \forall\,\alpha\in\mathcal{J}_n^{\delta_0}(\mu^*)).
$$
Hence, Hypothesis (T)  holds with $x_0=0$,  $\varepsilon_0=\min_{1\leq i \leq N}\text{d}_H(\varphi_i([0,1]),\{0,1\})>0$, $n_0=1$, $a=1$ and $c=2(N-1)$.
\end{proof}

\subsubsection{Completing the proof of Theorem~\ref{th:NPR2}}

By Theorem~\ref{th:main},  there is $\delta>0$ such that $f_\mu$ is asymptotically periodic on $Z_\mu$ for almost every $\mu\in U_\delta(\mu^*)$.  
By Lemma~\ref{lem:sing connections},  the set $U^*$ has full measure in $U$. Hence, since $\mu^*$ is an arbitrary point of $U^*$, we conclude, by a Lebesgue density point argument and by Lemma \ref{ppar}, that there exists a subset $U'\subset  U^*$ of full Lebesgue measure in $U$ such that for every $\mu\in U'$, $f_\mu$ is asymptotically periodic on $Z_\mu$ and its periodic points are regular. Moreover, since $(\Phi,\mu)$ has no singular connection, for every $x\in[0,1]$, there is $n\geq0$ such that $f^n_{\mu}(x)$ is a regular point, i.e., $f^n_{\mu}(x)\in Z_\mu$.  Putting it all together, in view of Remark~\ref{rem1}, we conclude that  $f_{\mu}$ is asymptotically periodic on $[0,1]$ for every $\mu\in U'$. This concludes the proof of Theorem~\ref{th:NPR2}. \qed

\section{Proof of Theorem \ref{thm:multidim}}\label{sec:proof multidim}

Given $\mu\in\Rr^\ell$, let $f_\mu$ be the piecewise-affine contraction on the compact convex polytope $X$ as defined in \eqref{fmu40}. 
Notice that, $\{f_\mu\}_{\mu\in \Rr^\ell}$ is a family of piecewise $\lambda$-contractions on $X$ (Definition \ref{def2}) with label set $\mathcal{A}=\{-1,1\}^\ell$, iterated function system $\{\varphi_i\colon\Rr^d\to\Rr^d\}_{i\in\mathcal{A}}$ of affine contractions $\varphi_i(x)=\Lambda_i x+b_i$ with $\lambda:=\max_i\|\Lambda_i\|<1$. 
The partition $\{A_{i,\mu}\}_{i\in\mathcal{A}}$ of $f_\mu$ varies with the parameter $\mu\in\Rr^\ell$ and it is defined by 
$$ A_{i,\mu} := \textrm{relint}(X)\cap \textrm{relint}\big(\{x\in X : \sigma_\mu(x) = i \}\big).
$$ 
Notice that $A_{i,\mu}$ can be the empty set.  Given $\alpha = (i_0,i_1,\ldots,i_{n-1})\in\mathcal{A}^n$ with $n\in\Nn$, we write $\varphi^\alpha(x):=\varphi_{i_{n-1}}\circ\cdots\circ\varphi_{i_0}(x)=\Lambda^\alpha x+ b_\alpha$ where $\Lambda^\alpha =\Lambda_{i_{n-1}}\cdots\Lambda_{i_0}$ and $b_\alpha = \varphi^\alpha(0)$. Let
\begin{equation}\label{eq919}
A_\mu^{\alpha}:=A_{i_0, \mu}\cap\varphi_{i_0}^{-1}(A_{i_1,\mu})\cap\cdots\cap (\varphi_{i_{n-2}}\circ\cdots\circ\varphi_{i_0})^{-1}(A_{i_{n-1}, \mu} ),
\end{equation}
as defined in \eqref{Qna2}. 

Let $V$ be the $\ell\times d$-matrix with rows $v^{(1)},\ldots, v^{(\ell)}$, where each $v^{(j)}$ is a unit vector orthogonal to the hyperplane $H_j(\mu)$ (see \eqref{Hj}), and let $C_i$ be the diagonal $\ell\times\ell$-matrix with diagonal equal to $i=(s_1,\ldots,s_\ell)\in\mathcal{A}$. \added{The following lemma gives a useful description of $A_\mu^\alpha$ in terms of the matrices $V$ and $C_i$.}
%\begin{lem}\label{lem:formula A mu alpha}
%For any $n\ge 1$, $\alpha=(i_0,i_1,\ldots,i_{n-1})\in \mathcal{A}^n$ and $\mu\in\Rr^\ell$, the set $A_{\mu}^{\alpha}$ is the open convex polytope given by
%$$
%A_\mu^\alpha = \bigcap_{j=0}^{n-1} \psi_j^{-1}(\mathrm{relint}(X))\cap \{x\in \Rr^d\colon C_{i_j} V \psi_j(x)>C_{i_j}\mu\},
%$$
%where $\psi_j:=\varphi_{i_{j-1}}\circ\cdots\circ\varphi_{i_0}$ for $1\le j\le n-1$ and $\psi_0=\mathrm{id}$. 
%\end{lem}
%\begin{proof} For any $\mu\in\Rr^\ell$ and  $i=(s_1,\ldots,s_\ell)\in\mathcal{A}$, the set $A_{i,\mu}$ is an open convex polytope since it is the intersection of the relative interior of the polytope $X$ with finitely many open half-spaces with bounding hyperplanes $H_j(\mu)$, $j=1,\ldots,\ell$. Indeed, we can write
%$$
%A_{i,\mu}=\text{relint}(X)\cap \left(\bigcap_{j=1}^\ell H^{s_j}_j(\mu)\right),
%$$
%where $i=(s_1,\ldots,s_\ell)$ and $H^{s}_j(\mu):=\left\{x\in\Rr^d\colon s\langle v^{(j)},x\rangle > s\mu_j\right\}$ for each $s\in\{-1,1\}$.
%Taking into account that
%$$
%\bigcap_{j=1}^\ell H^{s_j}_j(\mu)=\{x\in \Rr^d\colon C_i V x>C_{i}\mu\}
%$$
%the lemma follows from formula \eqref{eq919}. 
%\end{proof}

\begin{lem}\label{lem:formula A mu alpha}
For any $n\ge 1$, $\alpha=(i_0,i_1,\ldots,i_{n-1})\in \mathcal{A}^n$ and $\mu\in\Rr^\ell$, the set $A_{\mu}^{\alpha}$ is the open convex polytope given by
$$
A_\mu^\alpha = \mathrm{relint}(X)\cap\left(\bigcap_{k=0}^{n-1}  \{x\in \Rr^d\colon C_{i_k} V \varphi^{\alpha_k}(x)>C_{i_k}\mu\}\right),
$$
where $\alpha_k = (i_0,i_1,\ldots,i_{k-1})$ for $1\le k\le n-1$ and $\varphi^{\alpha_0}=\mathrm{id}$. 
\end{lem}
\begin{proof} For any $\mu\in\Rr^\ell$ and  $i=(s_1,\ldots,s_\ell)\in\mathcal{A}$, the set $A_{i,\mu}$ is an open convex polytope since it is the intersection of the relative interior of the polytope $X$ with finitely many open half-spaces with bounding hyperplanes $H_j(\mu)$, $j=1,\ldots,\ell$. Indeed, we can write
$$
A_{i,\mu}=\text{relint}(X)\cap \left(\bigcap_{j=1}^\ell H^{s_j}_j(\mu)\right),
$$ 
where $H^{s_j}_j(\mu):=\left\{x\in\Rr^d\colon s_j\langle v^{(j)},x\rangle > s_j\mu_j\right\}$ for each $s_j\in\{-1,1\}$.  Using the matrices introduced just before the statement of the lemma, we can write $A_{i,\mu}$ as follows
$$
A_{i,\mu}=\text{relint}(X)\cap\{x\in \Rr^d\colon C_i V x>C_{i}\mu\}.
$$
Now, given $n\ge 1$ and $0\le k\le n-1$, we get 
$$
(\varphi^{\alpha_k})^{-1}(A_{i_k,\mu}) = (\varphi^{\alpha_k})^{-1}(\text{relint}(X))  \cap \{x\in \Rr^d\colon C_{i_k} V \varphi^{\alpha_k}(x)>C_{i_k}\mu\}.
$$
Putting all together (see \eqref{eq919}),
$$
A_\mu^\alpha = \bigcap_{k=0}^{n-1} (\varphi^{\alpha_k})^{-1}(\mathrm{relint}(X))\cap \{x\in \Rr^d\colon C_{i_k} V \varphi^{\alpha_k}(x)>C_{i_k}\mu\}.
$$
In order to finish the proof, just notice that
$$
\mathrm{relint}(X)\cap (\varphi^{\alpha_k})^{-1}(\mathrm{relint}(X)) = \mathrm{relint}(X), \quad 0\leq k\leq n-1.
$$
Indeed, by assumption,  for all $j$, $\varphi_{i_j}(X)\subset \mathrm{relint}(X)$, which implies that $\varphi^{\alpha_k}(\mathrm{relint}(X))=\varphi_{i_{k-1}}\circ\cdots\circ\varphi_{i_0}(\mathrm{relint}(X))\subset  \mathrm{relint}(X)$. Hence, $\mathrm{relint}(X)\subset (\varphi^{\alpha_k})^{-1}(\mathrm{relint}(X))$.
\end{proof}

%Moreover, the singular set of $f_\mu$ is
%$$ S_\mu = \partial X \cup \left(\cup_{j=1}^{\ell} X\cap H_j(\mu)\right).$$
Let
$$
\mathcal{X}=\bigcup_{n\geq1}\mathcal{A}^n.
$$

Given $\alpha \in\mathcal{X}$, define  
 $$
Z_\alpha=\left\{\mu\in \Rr^\ell\colon A^\alpha_\mu\neq\emptyset \right\}.
$$
Clearly, $\mu\in Z_\alpha$ if and only if $\alpha\in \mathcal{I}_n(f_\mu)$ with $n$ being the length of the tuple $\alpha$ and $\mathcal{I}_n(f_\mu)$ being the set of itineraries of order $n$ of $f_\mu$ (Definition~\ref{def itineraries}). 
\begin{lem}\label{lem:convex Ualpha}
$Z_\alpha$ is convex for every $\alpha\in\mathcal{X}$.
\end{lem}
\begin{proof} 
Given any $\alpha=(i_0,i_1,\ldots,i_{n-1})\in\mathcal{A}^n$ with $n\geq1$, define
$$
\Pi_\alpha := \{(x,\mu)\in  \Rr^d\times \Rr^\ell\colon x\in A_\mu^\alpha\}.
$$
It follows from Lemma~\ref{lem:formula A mu alpha} that $(x,\mu)\in\Pi_\alpha $ iff $x\in \mathrm{relint}(X)$ and
$$
C_{i_j} V \varphi^{\alpha_j}(x)>C_{i_j}\mu,\quad\forall\,j\in\{0,\ldots,n-1\}.
$$
Notice that $\mathrm{relint}(X)$ is a convex set since the relative interior of a convex set is also convex. Now, we claim that $\Pi_{\alpha}$ is convex, i.e., the convex combination of any two points $(x_0,\mu_0)$ and $(x_1,\mu_1)$ of $\Pi_\alpha$ also belongs to $\Pi_\alpha$. To show that, let $t\in [0,1]$ and
$$(x_t, \mu_t) := \big( (1-t)x_0 + t x_1, (1-t) \mu_0 +t \mu_1 \big).
$$
By the hypotheses on $(x_0, \mu_0)$ and $(x_1, \mu_1)$, we have that for all $0\le j\le n-1$,
$$
 x_0, x_1\in \mathrm{relint}(X), \quad C_{i_j} V \varphi^{\alpha_j}(x_0)>C_{i_j}\mu_0\quad \textrm{and}\quad C_{i_j} V \varphi^{\alpha_j}(x_1)>C_{i_j}\mu_1.
$$
 Because $\mathrm{relint}(X)$ is convex, $x_t\in \mathrm{relint}(X)$ for every $t\in[0,1]$.  On the other hand, for all $0\le j\le n-1$ and $t\in[0,1]$,
\begin{align*}
C_{i_j}V \varphi^{\alpha_j}(x_t) &= (1-t) C_{i_j}V\varphi^{\alpha_j}(x_0)+ tC_{i_j} \varphi^{\alpha_j}(x_1)\\
&> (1-t) C_{i_j} \mu_0+ t C_{i_j}\mu_1\\
&= C_{i_j} \mu_t.
\end{align*}
This proves the claim. Finally, since $Z_\alpha$ is the projection of $\Pi_\alpha$ onto the second variable, we conclude that $Z_\alpha$ is also convex. 
\end{proof}

Let
$$
W_1:= \Rr^\ell\Big\backslash \left(\bigcup_{\alpha\in\mathcal{X}} \partial Z_\alpha\right).
$$

\added{During the proof of Theorem~\ref{thm:multidim}, we will construct a set of "good" parameters as the intersection of two subsets of $\mathbb{R}^\ell$. For parameters in this intersection, we will verify that hypotheses (E) and (T) are satisfied. The first of these subsets is the set $W_1$.}
\begin{lem}
$W_1$ is a full Lebesgue measure set. 
\end{lem}
\begin{proof}
Since the boundary of a convex set has  Lebesgue measure zero, we conclude from Lemma~\ref{lem:convex Ualpha} that $\partial Z_\alpha$ is a null set for every $\alpha\in\mathcal{X}$. This shows that $W_1$ has full measure.
\end{proof}
Hereafter, fix $\mu^*\in W_1$.  Let $U=B_1(\mu^*)$.  

Recall from \eqref{Qna2} that $\mathscr{C}_\mu^{(1)}=\{A_{i,\mu}\}_{i\in\mathcal{A}}$ and, for $n\ge 2$, $\mathscr{C}_\mu^{(n)}$ is the collection of non-empty open convex sets $A_\mu^\alpha$,
where $\alpha=(i_0,i_1,\ldots,i_{n-1})$ runs over the set $\mathcal{I}_{n}(f_{\mu})$.

The next two lemmas show that $\{\mathscr{C}_\mu^{(n)}\}_{\mu\in U}$ is stable at $\mu=\mu^*$.

\begin{lem}\label{Hausdorff A} Given $n\geq1$ and $\alpha\in \mathcal{I}_{n}(f_{\mu^*})$,  there exist  $c=c(\alpha)>0$ and $\delta=\delta(\alpha)>0$ such that if $\mu\in U_\delta(\mu^*)$, then $\alpha\in \mathcal{I}_{n}(f_{\mu})$ and
$$
d_{\mathrm{H}}(A_{\mu}^{\alpha},A_{\mu^*}^{\alpha})\leq c \|\mu-\mu^*\|.
$$
\end{lem}

\begin{proof}Let $n\ge 1$ and $\alpha=(i_0, i_1,\ldots, i_{n-1}) \in\mathcal{A}^n$.  By Lemma~\ref{lem:formula A mu alpha}, for every $\mu\in\Rr^\ell$,
$$
A_\mu^\alpha = \mathrm{relint}(X)\cap\left(\bigcap_{k=0}^{n-1}  \{x\in \Rr^d\colon C_{i_k} V \varphi^{\alpha_k}(x)>C_{i_k}\mu\}\right).
$$
By hypothesis, $\alpha\in\mathcal{I}_n(f_\mu^*)$, and so $A_{\mu^*}^\alpha\neq\emptyset$. Therefore, by continuity in $\mu$, there is $\delta=\delta(\alpha)>0$ such that $A_\mu^{\alpha}\neq\emptyset$ for every $\mu\in U_\delta(\mu^*)$, thus $\alpha\in \mathcal{I}_{n}(f_{\mu})$.

Now, for every $x\in A_{\mu^*}^{\alpha}$, we have
$$
d(x,\overline{A_{\mu}^{\alpha}}) 
=d\left(x,X\cap \left( \bigcap_{k=0}^{n-1}  \{x\in \Rr^d\colon C_{i_k} V \varphi^{\alpha_k}(x)\geq C_{i_k}\mu\}\right)\right)
= d\left(x,G_\mu\right),
$$
where 
$$
G_\mu = \{x\in \Rr^d\colon Mx\leq c,\,C_{i_k} V \varphi^{\alpha_k}(x)\geq C_{i_k}\mu,\, k=0,\ldots,n-1 \}
$$
and we have used the definition of $X$ in \eqref{defX}. By Theorem~\ref{thm distance to polyhedron}, there is a constant $c=c(\alpha)>0$ such that
$$
d(x,G_\mu)\leq c\| \mu-\mu^*\|.
$$
A similar estimate is valid for $d(x,G_{\mu^*})$ for any $x\in A_\mu^{\alpha}$. 
Taking the maximum of these two estimates, the lemma follows.
\end{proof}

\added{Using the previous lemma, we show that $\{\mathscr{C}_\mu^{(n)}\}_{\mu\in U}$ is stable at $\mu=\mu^*$, an  important property for verifying the hypothesis (E) (see Proposition \ref{lem26}).}
\begin{lem}\label{stable C2}  $\{\mathscr{C}_\mu^{(n)}\}_{\mu\in U}$ is stable at $\mu=\mu^*$.
\end{lem}

\begin{proof}
We claim that there is $\delta>0$ such that $\mathcal{I}_{n}(f_{\mu})=\mathcal{I}_{n}(f_{\mu^*})$ for every $\mu\in U_\delta(\mu^*)$. Indeed, by Lemma~\ref{Hausdorff A}, we have $\mathcal{I}_{n}(f_{\mu^*})\subset \mathcal{I}_{n}(f_{\mu})$ for every $\mu\in U_\delta(\mu^*)$ and some $\delta>0$. Now, suppose by contradiction, that there are sequences $(\mu_k)_{k\geq1}$ in $\mathbb{R}^{\ell}$ converging to $\mu^*$ and $(\alpha^{(k)})_{k\ge 1}$ in $\mathcal{A}^n$
 such that $\alpha^{(k)}\in \mathcal{I}_{n}(f_{\mu_k})$ but $\alpha^{(k)}\notin \mathcal{I}_{n}(f_{\mu^*})$ for every $k\geq1$. Since $\# \mathcal{I}_{n}(f_{\mu_k})\leq (\#\mathcal{A})^n$, by the pigeonhole principle, we may take a subsequence and assume that $\alpha^{(k)}$ is independent of $k$, hence we have $\alpha\in \mathcal{I}_{n}(f_{\mu_k})$ for every $k\geq1$ but $\alpha\notin \mathcal{I}_{n}(f_{\mu^*})$. This implies that $\mu_k\in Z_{\alpha}$ for every $k\geq 1$ and $\mu^*\notin Z_{\alpha}$. Thus, $\mu^* \in \partial Z_\alpha$, yielding  a contradiction. Therefore, the claim is true.
Now, using Lemma~\ref{Hausdorff A}, we conclude that for $\delta>0$ small enough and
 for every $\alpha\in\mathcal{I}_{n}(f_{\mu^*})$,  $A_\mu^{\alpha}$ converges to $A_{\mu^*}^{\alpha}$ as $\mu\to\mu^*$ in the Hausdorff metric. Hence,  $\{\mathscr{C}_\mu^{(n)}\}_{\mu\in U}$ is stable at $\mu^*$. 
\end{proof}

Next, we want to show that the multiplicity entropy of $f_{\mu^*}$ is zero. In order to do so, we need to further restrict $\mu^*$ inside another full Lebesgue measure set. Notice that piecewise-affine contractions may have positive multiplicity entropy \cite{KR06b}. We start with a fairly elementary result.

\begin{lem}\label{lem61}
Let $B\subset\mathbb{R}^d$ be a linearly independent set. Then, there is $\rho=\rho(B)>0$ such that for every $w \in \Rr^d$ with $\|w\|<\rho$ we have $w \notin \mathrm{aff}(B)$.
\end{lem}
\begin{proof}
Since $B$ is a set of linearly independent vectors, the affine hull $\mathrm{aff}(B)$ does not contain the origin of $\Rr^d$. Set $\rho:=d(0,\mathrm{aff}(B))$, then the conclusion follows.
\end{proof}

In what follows,  keep $1\leq j \leq\ell$ and $1\le k\le d$ unchanged. Let
$$
\mathcal{Y}_j^{(k)}=\left\{(\alpha_1,\ldots,\alpha_k)\in\mathcal{X}^k\colon \{(\Lambda^{\alpha_1})^Tv^{(j)},\ldots, (\Lambda^{\alpha_k})^T v^{(j)} \}\text{ is linearly independent}\right\}.
$$

For every $\boldsymbol{\alpha}=(\alpha_1,\ldots,\alpha_k)\in \mathcal{Y}_j^{(k)}$ and $\rho$ as in Lemma \ref{lem61}, we set  $$\rho_j(\boldsymbol{\alpha}):= \rho(B_j(\boldsymbol{\alpha})),\quad\textrm{where}\quad B_j(\boldsymbol{\alpha}):=\{(\Lambda^{\alpha_1})^Tv^{(j)},\ldots, (\Lambda^{\alpha_k})^T v^{(j)}\}.$$\vspace{-0.2cm}

Given $\boldsymbol{\alpha}=(\alpha_1,\ldots,\alpha_k)\in \mathcal{Y}_j^{(k)}$, let
$$
\mathcal{Z}_{j}(\boldsymbol{\alpha}) :=\left \{\beta\in \mathcal{X}\colon \|(\Lambda^\beta)^T\|<\rho_j(\boldsymbol{\alpha})\quad\mathrm{and}\quad (\Lambda^\beta)^Tv^{(j)}\in\mathrm{span}(B_j(\boldsymbol{\alpha})) \right\}.
$$

Given $\boldsymbol{\alpha}=(\alpha_1,\ldots,\alpha_k)\in \mathcal{Y}_j^{(k)}$ and $\beta\in \mathcal{Z}_{j}(\boldsymbol{\alpha}) $, let $a^{(1)}_{j,\beta}(\boldsymbol{\alpha}),\ldots, a^{(k)}_{j,\beta}(\boldsymbol{\alpha})$  be the coordinates of the vector $(\Lambda^\beta)^T v^{(j)}$ with respect to $B_j(\boldsymbol{\alpha})$, i.e., assume that
\begin{equation}\label{numbers}
(\Lambda^\beta)^T v^{(j)} = \sum_{q=1}^k a^{(q)}_{j,\beta}(\boldsymbol{\alpha}) (\Lambda^{\alpha_q})^T v^{(j)}.
\end{equation}
The following claim is true.\\

\noindent{Claim A.} $\sum_{q=1}^k a^{(q)}_{j,\beta}(\boldsymbol{\alpha})\neq 1$.\\

In fact, by Lemma~\ref{lem61}, we have $
(\Lambda^\beta)^T v_j\notin \mathrm{aff}(B_j(\boldsymbol{\alpha})), 
$
which proves the claim.

Given $1\leq j \leq\ell$, $1\leq k \leq d$, $\boldsymbol{\alpha}\in \mathcal{Y}_j^{(k)}$ and $\beta\in \mathcal{Z}_{j}(\boldsymbol{\alpha})$, let 
$$
\Gamma_{j,\beta}(\boldsymbol{\alpha}):=\left\{\mu \in \Rr^\ell\colon \mu_j = \frac{\langle b_\beta,v^{(j)}\rangle-\sum_{q=1}^k a^{(q)}_{j,\beta}(\boldsymbol{\alpha})\langle b_{\alpha_q},v^{(j)}\rangle}{1-\sum_{q=1}^k a^{(q)}_{j,\beta}(\boldsymbol{\alpha})}\right\}
$$
and
$$
W_2 :=  R^\ell\bigg\backslash\left(\bigcup_{j=1}^\ell\bigcup_{k=1}^d\bigcup_{\boldsymbol{\alpha}\in\mathcal{Y}_j^{(k)}}\bigcup_{\beta\in\mathcal{Z}_{j}(\boldsymbol{\alpha})}\Gamma_{j,\beta}(\boldsymbol{\alpha})\right).
$$
Notice that, by Claim A, $W_2$ is well-defined.
\begin{lem}
$ W_2$ is a full Lebesgue measure set.
\end{lem}
\begin{proof}
Since each set $\Gamma_{j,\beta}(\boldsymbol{\alpha})$ is a hyperplane in $\Rr^\ell$ and $W_2$ equals the whole $ \Rr^\ell$ but a countable union of such hyperplanes, it follows that $W_2$ has full Lebesgue measure.
\end{proof}

Given $\mu\in\Rr^\ell$ and $n\geq1$, define the following collection of hyperplanes
$$
\mathcal{H}^{(n)}_\mu:=\{(\varphi^\alpha)^{-1}(H_j(\mu))\colon  \alpha\in \mathcal{A}^k, \, 0\leq k\le n-1,\, 1\leq j\leq \ell\},
$$
where $\varphi^0=\mathrm{id}$ is the identity map and $\mathcal{A}^{0}:=\{0\}$. Denote by $\mathcal{D}^{(n)}_\mu$  the collection of  connected components of
$
\Rr^d\setminus \bigcup_{H\in \mathcal{H}^{(n)}_\mu}H.
$ 
Notice that $\text{mult} (\mathcal{D}^{(n)}_\mu)$ is bounded by the number of regions of a central hyperplane arrangement \added{(see \cite[Corollary, p. 816]{Winder})}, i.e.,
\begin{equation}\label{1070}
\text{mult} (\mathcal{D}^{(n)}_\mu)\leq 2\sum_{k=0}^{d-1} {m_\mu^{(n)}-1 \choose k},
\end{equation}
where 
$$
m_\mu^{(n)}:=\max_{x\in\Rr^d}\#\{H\in \mathcal{H}^{(n)}_\mu\colon x\in H\}.
$$

\begin{lem}\label{lem:m bounded} For every $\mu\in W_2$, the sequence $(m_\mu^{(n)})_{n\geq1}$ is bounded. 
\end{lem}

\begin{proof}
Suppose, by contradiction, that there is $\mu\in W_2$ such that $(m_\mu^{(n)})_{n\ge 1}$ is an unbounded sequence. Thus, there is an increasing integer sequence $(n_p)_{p\in\Nn}$  such that $m_\mu^{(n_p)}\geq p$ for every $p\in\Nn$. For each $p\in\Nn$, let $x_p\in\Rr^d$ be such that \mbox{$\#\{H\in \mathcal{H}^{(n_p)}_\mu\colon x_p\in H\}\geq p$}. Using the pigeonhole principle and taking $p$ sufficiently large, there exist $1\leq j \leq\ell$, $1\leq k \leq d$, $\boldsymbol{\alpha}=(\alpha_1,\ldots,\alpha_k)\in\mathcal{Y}_j^{(k)}$ and $\beta\in\mathcal{Z}_j(\boldsymbol{\alpha})$ such that
$$
\begin{cases}
\langle \varphi^{\alpha_q}(x_p),v^{(j)}\rangle = \mu_j,\quad q=1,\ldots,k\\
\langle \varphi^\beta(x_p),v^{(j)} \rangle = \mu_j.
\end{cases}
$$
The previous system is equivalent to 
\begin{equation}\label{10102024}
\begin{cases}
\langle x_p,(\Lambda^{\alpha_q})^Tv^{(j)}\rangle = \mu_j-\langle b_{\alpha_q},v^{(j)}\rangle,\quad q=1,\ldots,k\\
\langle x_p,(\Lambda^\beta)^Tv^{(j)}\rangle = \mu_j-\langle b_\beta,v^{(j)}\rangle.
\end{cases}
\end{equation}
By Claim A, $(\Lambda^\beta)^T v^{(j)} = \sum_{q=1}^k a^{(q)}_{j,\beta}(\boldsymbol{\alpha}) (\Lambda^{\alpha_q})^Tv^{(j)}$ with $\sum_{q=1}^k a^{(q)}_{j,\beta}(\boldsymbol{\alpha})\neq1$. Then, by \eqref{10102024}, we obtain
$$
\mu_j\left(1-\sum_{q=1}^k a^{(q)}_{j,\beta}(\boldsymbol{\alpha})\right)=\langle b_\beta,v_j\rangle-\sum_{q=1}^k a^{(q)}_{j,\beta}(\boldsymbol{\alpha})\langle b_{\alpha_q},v_j\rangle,
$$
which implies that $\mu\in\Gamma_{j,\beta}(\boldsymbol{\alpha})$, yielding a contradiction.
\end{proof}

Using the previous lemma and the fact that $\text{mult} (\mathscr{C}_\mu^{(n)}) \leq \text{mult} (\mathcal{D}^{(n)}_\mu)$, we get the following result.
\begin{lem}\label{lem59}
$H_{\mathrm{mult}}(f_{\mu})=0$ every $\mu\in W_2$.
\end{lem}

Set
$$
W=W_1\cap W_2. 
$$

With the next two lemmas, in light of Theorem~\ref{th:main}, we conclude the proof of Theorem~\ref{thm:multidim}.

\begin{lem}
For every $\mu^*\in W$, $\{f_\mu\}_{\mu\in U}$ satisfies Hypothesis (E) at $\mu=\mu^*$.
\end{lem}
\begin{proof} By Proposition \ref{lem26}, the claim follows from Lemma \ref{stable C2} and Lemm~\ref{lem59}.
\end{proof}

\begin{lem}
For every $\mu^*\in W$, $\{f_\mu\}_{\mu\in U}$ satisfies Hypothesis (T) at $\mu=\mu^*$.
\end{lem}
\begin{proof} Given $\epsilon>0$ and $\alpha\in\mathcal{X}$,   we have 
\begin{align*}
\text{Leb}^*\big(\left\{\mu\in U\colon d(\varphi^{\alpha}(0),S_{\mu}) \leq \varepsilon   \right\}\big)
 & \leq \sum_{j=1}^\ell \text{Leb}^*\big(\left\{\mu\in U \colon d(\varphi^\alpha(0),H_j(\mu))\leq \varepsilon\right\}\big)\\
 & = \sum_{j=1}^\ell \text{Leb}^*\big(\left\{\mu\in U\colon |\langle \varphi^\alpha(0),v^{(j)}\rangle -\mu_j|\leq \varepsilon\right\}\big)    \\\
 &\leq 2^d\ell\varepsilon.
\end{align*}
Thus, Hypothesis $(T)$ holds true provided $x_0=0$, $\varepsilon_0=1$, $\delta_0=1$, $n_0=1$, $a=1$ and $c=2^d\ell$. 
\end{proof}

\section{Proof of Corollary~\ref{cor:3}} \label{proofs corollaries}

Assume that $\varphi_i$, $i\in\mathcal{A}$, are contractive homotheties, i.e., there exist $\lambda_i\in (0,1)$, $b_i\in\mathbb{R}^d$, $i\in\mathcal{A}$, such that $\varphi_i(x) = \lambda_i x + b_i$ for all $x\in\mathbb{R}^d.$ 
Given $\alpha= (i_0, i_1,\ldots, i_{n-1})\in\mathcal{A}^n$, $n\ge 1$,  the map $\varphi^{\alpha}$ defined by \eqref{4344} satisfies
$$ \varphi^{\alpha}(x) = \lambda^\alpha x + b^{\alpha},\quad\forall x\in\mathbb{R}^d,
$$
where $b^\alpha =\varphi^\alpha(0)$ and $\lambda^\alpha=\lambda_{i_{n-1}}\cdots\lambda_{i_0}$. Let
$$
N_\alpha = \left\{\mu\in \Rr^\ell \colon \exists\,j\in\{1,\ldots,\ell\},\, \mu_j = \frac{\left\langle  v^{(j)}, b^{\alpha}\right\rangle}{1-\lambda^\alpha}\right\}.
$$
The set $N_\alpha$ is a union of $\ell$ hyperplanes. Thus, $N_\alpha$ has zero Lebesgue measure. 
%Let $N_\alpha\subset\mathbb{R}^{\ell}$ be set formed by the vectors $\mu=(\mu_1,\ldots,\mu_{\ell})\in\mathbb{R}^{\ell}$ such that
%$\mu_j = \left\langle  v^{(j)}, b^{\alpha}\right\rangle\big/(1-\lambda^\alpha)$ for some $1\le j\le \ell$. 
By Theorem~\ref{thm:multidim}, there exists a full Lebesgue measure set $W\subset\mathbb{R}^{\ell}{\setminus}\left( \bigcup_{n\geq1}\bigcup_{\alpha\in\mathcal{A}^n}N_\alpha\right)$ such that $f_\mu$ is asymptotically periodic on $Z(f_\mu)$ for all
$\mu\in W$. Hereafter, let $\mu\in W$ be fixed. In light of Remark~\ref{rem1}, either $f_\mu$ is asymptotically periodic on the entire domain $X$ or there exists $x\in X$ such that $\bigcup_{n\ge 0}\{f^n_\mu(x)\}\cap Z(f_\mu)=\emptyset$. Assume that this second alternative happens, i.e., $\bigcup_{n\ge 0}\{f^n_\mu(x)\}\subset X{\setminus} Z(f_\mu)$.
By hypothesis, $\varphi_i(X)\subset \mathrm{relint}(X)$ for every $i\in\mathcal{A}$. Hence, by the pigeonhole principle, there exists
$1\le j\le \ell$ such that the forward \mbox{$f_\mu$-orbit} of $x$ intersects the hyperplane $H_j(\mu)$ infinitely many times. In particular, there exist $n\ge 1$ and $y\in H_j(\mu) $ such that $f^n_\mu(y)\in H_j(\mu)$. Algebraically, 
$$ \left\langle v^{(j)}, y \right\rangle = \mu_j =  \left\langle v^{(j)}, f_\mu^n(y) \right\rangle.
$$
Since $f_\mu^n(y) = \lambda^\alpha y + b^{\alpha}$ for some $\alpha\in \mathcal{A}^n$, we obtain
$$ 
\mu_j =  \left\langle v^{(j)}, f_\mu^n(y) \right\rangle =  \left\langle v^{(j)}, \lambda^\alpha y + b^{\alpha} \right\rangle = \lambda^\alpha \mu_j +  \left\langle v^{(j)}, b^{\alpha} \right\rangle, 
$$
that is, $\mu_j = \left\langle  v^{(j)}, b^{\alpha}\right\rangle\big/(1-\lambda^\alpha)$, contradicting the fact that $\mu\not\in N_\alpha$. In this way, only the first alternative happens and $f_\mu$ is asymptotically periodic on the entire domain $X$.
A similar argument shows that the orbits of $f_\mu$ are regular.\qed

\appendix
\section{Distance to a polyhedron}
Let $A$ be an $m\times n$ matrix and $\beta=\beta(A)$ be the least number such that for each non-singular submatrix $B$ of $A$ all entries of $B^{-1}$ are at most $\beta$ in absolute value. Given $b\in\Rr^m$, let $G_b$ be the polyhedron
$$
G_b=\{x\in\Rr^n\colon Ax\leq b\}.
$$
Denote by $\|\cdot\|_{\infty}$ the $\ell^{\infty}$-norm.

\begin{thm}[Theorem 0.1 in \cite{MR1158365}]\label{thm distance to polyhedron} Given $b_0\in\Rr^m$ and $b\in\Rr^m$, suppose that $G_{b_0}\neq\emptyset$ and $G_{b}\neq\emptyset$, respectively. For every $x_0 \in G_{b_0}$,
$$
\inf_{x\in G_b} \|x-x_0\|_{\infty}\leq n \beta \|b-b_0\|_{\infty}.
$$
\end{thm}

%%%%%%%%%%%%%%%%%%%%%%%%%%%%%%%%%%%%%%%%%%%%%%%%%%%%%%%%%%%%%%%%%%%%%%%%%%%%

\bibliographystyle{amsplain}
\bibliography{Bibfile}

\end{document}